\documentclass[oneside]{amsart}
\usepackage{amsmath}
\usepackage{amsfonts}
\usepackage{amssymb}
\usepackage{latexsym}
\usepackage{amsthm}
\usepackage{diagrams}
\usepackage{setspace}
\usepackage{hyperref}
\usepackage{geometry}
\usepackage{xypic}
\usepackage[usenames,dvipsnames]{color}

\everymath{\displaystyle}
\usepackage{mymacros}
\usepackage{url}

\newcommand{\OR}{\operatorname{\textbf{OR}}}
\newcommand{\CM}{\operatorname{\textbf{CM}}}
\newcommand{\NCM}{\operatorname{\textbf{NCM}}}
\newcommand{\Nm}{\operatorname{\textbf{N}}}

\begin{document}

\title{A possible generalization of Maeda's conjecture}

\author{Panagiotis Tsaknias}
\address{Department of Mathematics, University of Luxembourg,
                          Campus Kirchberg, 6 rue Richard Coudenhove-Kalergi, L-1359 Luxembourg}
\email{panagiotis.tsaknias@uni.lu}
\email{p.tsaknias@gmail.com}

\date{May 2012}

\begin{abstract}
We report on observations we made on computational data that suggest a generalization of Maeda's conjecture regarding the number of Galois orbits of newforms of level $N = 1$, to higher levels. They also suggest a possible formula for this number in many of these cases.
\end{abstract}

\maketitle

\section{Introduction}

For $N\geq 1$ and $k\geq 2$ even let $S_k(N)$ denote the space of modular forms of weight $k$ and level $\Gamma_0(N)$ and let $S_k^{\textrm{new}}(N)$ be its new subspace. In the case of $N=1$ Maeda made the following conjecture:

\begin{conj}
The characteristic polynomial of the Hecke operator $T_p$ when acting on $S_k(1)$ is irreducible for all primes $p$ and for all weights $k$, provided the space is non-trivial.
\end{conj}

This conjecture implies the following, weaker one:

\begin{conj}\label{cj:maeda2}
$S_k(1)$ consists of a single newform Galois orbit for all $k$.
\end{conj}

We believe we can propose a few generalizations of the latter one, based on computational data we have accumulated and that we will present in the next section:

\begin{itemize}
\item Fix $N$. Then the number of newform Galois orbits in $S_k^{\textrm{new}}(N)$, as a function of $k$, is bounded.
\item Fix $N$. The number of non-CM newform Galois orbits in $S_k^{\textrm{new}}(N)$ as a function of $k$ is eventually constant.
\item Assuming the above guess is true, let $\NCM(N)$ be the constant to which the number of non-CM newform orbits in $S_k^{\textrm{new}}(N)$ eventually stabilizes. Then $\NCM()$ is a multiplicative function, i.e. 
$$\NCM(NM) = \NCM(N)\NCM(M)$$
for all coprime integers $N$ and $M$.
\end{itemize}

In Section \ref{sec:notations} we introduce notation and formulate our precise questions. In Section \ref{sec:NCM} we have brief discussion on the non-CM newform Galois orbits, which is the mysterious part, and we continue in Section \ref{sec:CM} with an overview of the theory of Hecke characters and CM forms, which we use to get a grasp on the number of CM newform Galois orbits. Finally in Section \ref{sec:observations} we give the table with the patterns we have observed that led us to formulate our questions and in Appendix \ref{app:Data} we give tables with all the computational data on which all our present questions are based.

\section{Setup-Notation}\label{sec:notations}

Let $N,k$ be positive integers, $k$ even. We will denote by $S_k(N)$ the space of cusp forms of level $\Gamma_0(N)$ and weight $k$ and by $S_k^{\textrm{new}}(N)$ its new-subspace. Every element $f$ of $S_k(N)$ has a Fourier expansion $\sum_{n\geq1}c_n(f)q^n$ which is called the $q$-expansion of $f$. The space $S_k(N)$ admits an action by the $\Z$-algebra of the Hecke operators $T_n$ which preserves the new subspace. We will denote this algebra by $\T_k(N)$ and by $\T_k^{\text{new}}(N)$ its quotient through which it acts on the new subspace. Finally $S_k^{\text{new}}(N)$ admits a basis of simultaneous eigenvectors for all the operators of $\T_k(N)$, which we can assume to be normalized, i.e that $c_1(f)=1$ for each of these eigenvectors $f$.

It turns out for such a normalized eigenvector $f$ that the extension $K_f:= \Q(\{c_n(f)\}_n)/\Q$ is finite, and that the coefficients $c_n(f)$ are algebraic integers. In view of this we have an action of $\gq:=\GQ$ on these $f$ that commutes with the one of $\T_k(N)$ and therefore it preserves a given basis of normalized eigenforms. It also actually preserves a basis of normalized eigenforms of $S_k^{\text{new}}(N)$.
We will call the orbits of this last action the newform Galois orbits of level $N$ and weight $k$, and denote their number by $\OR(N,k)$. 

The purpose of this article is to report on computations regarding $\OR(N,k)$ for $N > 1$.  We provide tables with all the data that we used to formulate the following questions at the end of this article. We need to introduce some other notions and notation first.

Let $\e:(\Z/D\Z)^\times\lra\C^\times$ be a Dirichlet character. If $f$ is a modular form of level $N$ and weight $k$ then
$$f\otimes\e:=\sum_{n\geq1}\e(n)c_n(f)q^n$$
is also a modular form of level (dividing) $ND^2$ and weight $k$ (but not necessarily of trivial Nebentypus, even if $f$ is!).
If $f=f\otimes\e$ for some nontrivial quadratic character $\e$ then $f$ is called a complex multiplication form (CM form for short). Notice that if $f$ is a CM form then $f^\sigma$ is also a CM form for all $\sigma \in \gq$. Therefore we can talk about CM and non-CM newform Galois orbits. For a fixed level $N$ and weight $k$ we will denote their number by $\CM(N,k)$ and $\NCM(N,k)$ respectively. Obviously the following holds:
$$\OR(N,k) = \CM(N,k) + \NCM(N,k).$$
Looking through the tables we have compiled in Section \ref{sec:observations} we were led to formulate the following questions:
\begin{qst}\label{qst:OR}
Fix an integer $N\geq1$. Then $\OR(N,k)$, as a function of $k$, is bounded.
\end{qst}
\begin{qst}\label{qst:NCM}
Fix an integer $N\geq1$. Then $\NCM(N,k)$, as a function of $k$, is eventually constant. If we let $\NCM(N)$ be this constant then we have that:
$$\NCM(NM)=\NCM(N)\NCM(M)\qquad\forall(N,M)=1,$$
i.e. $\NCM()$ is a multiplicative function.
\end{qst}
Specializing this last question to squarefree levels we get:
\begin{qst}\label{qst:sqfree}
Fix an integer $N=p_1\cdots p_s\geq1$, where the $p_i$'s are pairwise distinct prime numbers and $s\geq0$. Then the number of newform Galois orbits in $S_k^{\textrm{new}}(N)$ is eventually equal to $2^s$.
\end{qst}
The tables of Section \ref{sec:observations} suggest that $\NCM(p)=2$ for every prime $p$ so a positive answer to Question \ref{qst:NCM} along with the fact that there aren't any CM newforms for squarefree levels (see Section \ref{sec:CM} for details), would imply a positive answer to Question \ref{qst:sqfree}. Of course one can get many more conjectural formulas like the one of Question \ref{qst:sqfree} for levels containing higher prime powers by combining the suggested multiplicativity of $\NCM$ along with observations from the tables of Section \ref{sec:observations}. Observing for  example that $\NCM(4)=1$ and, as before, that $\NCM(p)=2$ for every prime, then assuming Question \ref{qst:NCM} is true we get:
$$\NCM(2^2p)=2\qquad\textrm{for all primes }p>2.$$

\section{Non-CM forms}\label{sec:NCM}

In the early versions of this manuscript we had only included Question  \ref{qst:OR} along with the observation that, when stable, the number of newform Galois orbits seemed to behave multiplicatively in many cases (like the squarefree level case for example, which is the subject of Question \ref{qst:sqfree}).
After we emailed these to William Stein he informed us that he had conducted similar computations and suggested that one should look separately the cases of CM and non-CM forms.
Moreover he suggested that the decomposition of the non-CM part should be forced by the Atkin-Lehner operators.

In the squarefree level case this agrees with Question \ref{qst:sqfree}: As we have already mentioned, there are no CM forms for squarefree levels. The number of Atkin-Lehner operators is $s$, where $N=p_1\cdots p_s$ and therefore $2^s$
possible combinations of eigenvalues for all the operators considered together. It turns out that each combination does occur for exactly one non-CM Galois orbit.

Conceptually, the Atkin-Lehner operators seem to be a good candidate to explain that $\NCM(N,k)$ is (eventually) independent of $k$ since they only depend on the level and they are the ones that cut out the newspace out of $S_k(N)$.

For non-squarefree levels though things get more complicated. The Atkin-Lehner operators are not enough to distinguish between all non-CM Galois orbits, i.e. the same combination of Atkin-Lehner eigenvalues occurs for more than one non-CM Galois orbit. For example, in the case $N=16=2^4$ the computational data suggest that there are eventually $6$ newform Galois orbits for all sufficiently big weights $k$, none of which is CM. there is only one Atkin-Lehner operator, $w_{16}$, which cuts the newsubspace into two pieces. These two pieces break further into $3$ irreducible ones each,  which correspond to the newform Galois orbits. This suggests that a further refinement of the Atkin-Lehner decomposition of $S_k^{\textrm{new}}(N)$ is needed in order to distinguish all the different non-CM orbits, but we are unable however to determine how it should be obtained.

On the other hand however, the distinction between CM and non-CM Galois orbits seems essential: The CM Galois orbits seem to be the reason $\OR(N,k)$ is eventually periodic (whenever this happens) instead of  eventually constant as a function of $k$ and seem to be also the ``error term" that is obstructing the number of Galois orbits to behave multiplicatively.

\section{CM forms}\label{sec:CM}

In this section we describe the connection between CM forms and Hecke characters of imaginary quadratic fields (hence the name CM forms). Our goal is to eventually explain why $\CM(N,k)$ is bounded and periodic. Most of the material in this section is well known facts that we apply in the case of interest to us, i.e. CM forms of trivial Nebentypus (and therefore even weight). The interested reader should consult \cite[\S 3]{Ribet1977}, \cite[\S 4.8]{Miyake2006} and \cite[Chapter VII, \S 3]{Neukirch1999}.

Let $K$ be an imaginary quadratic field, $\Oc=\Oc_K$ its ring of integers and $\mg$ an ideal of $K$. Let us also denote by $J^\mg$ the group of fractional ideals of $\Oc_K$ that are coprime to $\mg$. A Hecke character $\psi$ of $K$, of modulus $\mg$, is then a group homomorphism $\psi:J^\mg\lra \C^*$ such that there exists a character $\psi_f:(\Oc_K/\mg)^*\lra\C^*$ and a group homomorphism $\psi_\infty:K_\R^*\lra\C^*$, where $K_\R:=K\ox_\Q\R$, such that:
\begin{equation}\label{eq:comp}
\psi((\al))=\psi_f(\al)\psi_\infty(\al)\qquad\textrm{for all }\al\in K.
\end{equation}
One calls $\psi_f$ and $\psi_\infty$ the finite and the infinite type of $\psi$ respectively. The conductor of $\psi$ is the conductor of $\psi_f$ and we will call $\psi$ primitive if the conductor is equal to its modulus.
The relation \eqref{eq:comp} implies that $\psi_f(\e)\psi_\infty(\e)=1$ for all units $\e\in\Oc^*$. The converse is also true: Given a pair $(\psi_f,\psi_\infty)$ satisfying this relation for all units $\e\in\Oc^*$ then there is a Hecke character they are attached to it (\cite[Chapter VII, \S 3, Exercise 5]{Neukirch1999}). This Hecke character however might not be unique:
\begin{lem}\label{lem:Hecke}
The number of different Hecke characters associated with a given pair $(\psi_f,\psi_\infty)$ is equal to the class number $h_K$ of $K$.
\end{lem}
\begin{proof}
Let $\psi:J^\mg\lra\C^*$ be a Hecke character associated with the given pair. Let $P^\mg$ denote the subgroup of $J^\mg$ consisting of principal ideals. Then $J^\mg/P^\mg\isom J/P$, the class group of $K$. This is an abelian group so by the structure theorem of abelian groups it is isomorphic to a product $\prod_{i=1}^s C_{t_i}$ of cyclic groups with $\prod_it_i=h_K$. For every $i$ pick an ideal $\mathfrak{a}_i$ generating the group $C_{t_i}$. Then $\mathfrak{a}_i^{t_i}$ is a principal ideal $(\pi_i)$. We therefore have:
$$\psi^{t_i}(\mathfrak{a}_i)=\psi(\mathfrak{a}_i^{t_i})=\psi((\pi_i))=\psi_f(\pi_i)\psi_\infty(\pi_i).$$
This means that for every generator $\mathfrak{a}_i$ we have $t_i$ choices and therefore we have $\prod_it_i=h_K$ different Hecke characters associated with the same pair $(\psi_f,\psi_\infty)$.
\end{proof}
This reduces the problem of counting Hecke characters to counting compatible pairs $(\psi_f,\psi_\infty)$.

From this point on we specialize to the case where $K$ is an imaginary quadratic field. We will denote by $D$ its discriminant and by $(D|.)$ the Kronecker symbol associated to it. The only possible group homomorphisms $\psi_\infty$ are then of the form $\sigma^u$, where $\sigma$ is one of the two conjugate complex embeddings of $K$ and $u$ is a non negative integer. The following theorem associates a (CM) modular form a Hecke character $\psi$ of $K$ (\cite[Theorem 4.8.2]{Miyake2006}):
\begin{thm}\label{thm:Miyake}
Given $\psi$ of infinity type $(u,0)$ and  finite type $\psi_f$ with modulus $\mg$. Assume $u>0$ and let $\Nm\mg$ be the norm of $\mg$. We then have that
$$f=\sum_n(\sum_{\mathfrak{a}=n}\psi(\mathfrak{a}))q^n$$
is a cuspidal eigenform in $S_{u+1}(N,\chi)$, where $N=|D|\Nm\mg$ and $\chi(m)=(D|m)\psi_f(m)$ for all integers $m$.
\end{thm}

The eigenform associated with the Hecke character is new if and only if $\psi$ is primitive (see Remark 3.5 in \cite{Ribet1977}).

Let us introduce some terminology: We will say that a Hecke character $\psi$ is $(N,k)$-suitable if it satisfies the following:
\begin{itemize}
\item $\psi$ is primitive.
\item $|D|\Nm\mg=N$, where $D$ is the discriminant of $K$ and $\mg$ is the modulus of $\psi$.
\item The infinity type of $\psi$ is $\sigma^{k-1}$.
\item $\psi_f(m)=(D|m)$ for all integers $m$.
\end{itemize}
We will also say that a compatible pair $(\psi_f,\psi_\infty)$ is $(N,k)$-\emph{suitable}, if the Hecke characters attached to it are $(N,k)$-suitable.
An obvious consequence of Theorem \ref{thm:Miyake} is that the number of CM newforms in $S_k(\Ga_0(N))$ is equal to the number of $(N,k)$-suitable Hecke characters. Since the number of 
$(N,k)$-suitable Hecke characters is determined in view of Lemma \ref{lem:Hecke} by the number of $(N,k)$-suitable compatible pairs, we will focus on counting these instead. It is obvious that given a weight $k$ and a level, there is only one infinity type that could give an $(N,k)$-suitable pair, namely $\sigma^{k-1}$. With this in mind we are left with counting the number of finite types $\psi_f$ that are compatible with $\sigma^{k-1}$ and give an $(N,k)$-suitable pair.

First of all, notice that only a finite number of imaginary quadratic fields can give an $(N,k)$-suitable compatible pair: The discriminant $D$ of $K$ must divide $N$. Let $\mg$ be an ideal of $K$ such that $|D|\Nm\mg=N$. This allows only a finite number of ideals as the modulus of an $(N,k)$-suitable pair. 

\remark: Notice so far that the bound on the number of possible fields and the bound on the number of possible ideals for every possible field depend only $N$. This already gives a bound on the number of CM newforms of level $N$ and weight $k$ that is independent of $k$.

Let now $\psi_f:(\Oc_K/\mg)^*\lra\C^*$ be a finite type of modulus $\mg$, where $\mg$ is an ideal of $K$ such that $|D|\Nm\mg=N$. Consider the composition $\Z\lra(\Oc_K/\mg)\lra\C$ where the second map is $\psi_f$. The $(N,k)$-suitability of $(\psi_f,\psi_\infty)$ then implies that this composition is the Kronecker symbol $(D|.)$ which has conductor $|D|$ and we thus get an injection:
$$\Z/|D|\Z\lra(\Oc_K/\mg).$$
This implies in particular that $|D|$ divides $\Nm\mg$ and therefore that $D^2\mid N$, restricting further the possible $(N,k)$-suitable compatible pairs. 

The group $(\Oc/\mg)^*$ is an abelian group whose structure can be determined for example by the theorems in \cite{Nakagoshi1979}. These imply  in particular that $(\Z/|D|\Z)^*$ can be seen as a subgroup of this group and, since it is abelian, as a quotient of it as well. Since the finite type is completely determined on integers by the Kronecker symbol, to completely describe it one only has to describe its values on the generators of $(\Oc/\mg)^*/(\Z/|D|\Z)^*$ which is an abelian group too. It turns out that one has $\frac{\#(\Oc/\mg)^*}{\varphi(|D|)}$ different finite types whose values on integers match those of the Kronecker symbol. Out of all these choices, one must consider only those for which $\psi_f$ is of conductor $\mg$, in order for the corresponding cuspforms to be new (see Remark 3.5 in \cite{Ribet1977}).

The last condition one has to check is the compatibility with the infinity type on the integral units of $K$. Recall that this implies:
$$\psi_f(\e)=\psi_\infty^{-1}(\e)=\sigma^{1-k}(\e)\qquad\textrm{for all }\e\in\Oc^*_K$$
All imaginary quadratic fields contain at least the following integral units: $\{1,-1\}$. In both cases the compatibility condition is automatically satisfied no matter what the weight is. We stress one more time that up to this point, our treatment of possible finite types is independent of the weight $k$. The only imaginary quadratic fields $K$ that contain more integral units are the following:
\begin{itemize}
\item $K=\Q(\sqrt{-1})$. In this case $O_K$ contains 4 integral units: $\{1,-1, i, -i\}=\langle i\rangle$. We have that $\psi^2_f(i)=\psi_f(i^2)=(D|-1)=-1$, so $\psi_f(i)=\pm i$. On the other hand $\psi_f(i)=\sigma(i)^{1-k}=i^{1-k}$. If $\psi_f(i)=i$ we then get that $k\equiv0\mod4$. If $\psi_f(i)=-i$ we get that $k\equiv2\mod4$.
\item $K=\Q(\sqrt{-3})$. In this case $O_K$ contains 6 integral units: $\{1,-1,w,w^2,-w,-w^2\}=\langle-w\rangle$, where $w$ is a primitive cubic root of unity. We have that
$\psi^3_f(-w)=\psi_f((-w)^3)=(D|-1)=-1$, so $\psi_f(-w)=w^s$, where $s=0,1,2$. On the other hand $\psi_f(-w)=\sigma(-w)^{1-k}=(-w)^{1-k}$.\end{itemize}

We thus come to the following conclusion:
\begin{prop}
Fix a positive integer $N$ and let $k$ be a varying even integer. Then:
\begin{itemize}
\item The number of CM newforms in $S_k^{\textrm{new}}(N)$ with complex multiplication field $\Q(\sqrt{-d})$, where $d\neq 1,3$, is independent of the weight $k$, i.e. it is constant as a function of $k$.
\item The number of CM newforms in $S_k^{\textrm{new}}(N)$ with complex multiplication field $\Q(\sqrt{-1})$, depends only on $k\mod4$, i.e. has at most  two possible values and it is periodic as a function of $k$.
\item The number of CM newforms in $S_k^{\textrm{new}}(N)$ with complex multiplication field $\Q(\sqrt{-3})$, depends only on $k\mod6$, i.e. has at most  three possible values and it is periodic as a function of $k$.
\end{itemize}
\end{prop}
Two immediate corollaries are the following:
\begin{cor}
Fix a positive integer $N$. The number of CM newforms in $S_k^{\textrm{new}}(N)$ is a bounded function of $k$.
\end{cor}
\begin{cor}
Fix a positive integer $N$. Then $\CM(N,k)$ is a bounded function of $k$.
\end{cor}
\remark Hopefully it is clear by how we counted the possible CM newforms for a given level and weight that the way CM newforms group in Galois orbits depends on the degree of the field extension containing the values of $\psi_f$ and on the extension defined by the generating values for $\psi$ mentioned in the proof of Lemma \ref{lem:Hecke}. The variety of possible combinations however is preventing us from giving a precise recipe for the orbits or just their number.

In the next subsection we study the case $N=p^2$ where $p$ is a prime number. We describe the possible number of CM newforms and we are also able to explain how they group into Galois orbits due to the simplicity of this situation.

\subsection{$N=p^2$, $p$ prime:}

Let $\psi$ be a $(p^2,k)$-suitable Hecke character of the imaginary quadratic field $\Q(\sqrt{-d})$, of discriminant $D$, with modulus $\mg$. As we have seen in the previous section, the $(p^2,k)$-suitability of $\psi$ implies that $D^2|N$. Since
$$D=\Big\{\begin{matrix}-d&d\equiv3\mod4\\-4d&\textrm{otherwise}\end{matrix}$$
we get that $D=-p$ and $p\equiv3\mod4$. So if $p\not\equiv3\mod4$, then $\CM(p^2,k)=0$ for all even $k$.

For the rest of the section we therefore assume $p\equiv3\mod4$, $d=p$ and $D=-p$. The $(p^2,k)$-suitbility of $\psi$ implies that $\Nm\mg=p^2/|D|=p$. In particular $\mg$ is the prime ideal of $\Q(\sqrt{-p})$ above $p$. Since $\Nm\mg=\#(\Oc/\mg)$ and $\Oc/\mg$ is an extension of $\F_p$ we get that $\Oc_K/\mg\isom\F_p$ and in particular $(\Oc_K/\mg)^*\isom(\Z/p\Z)^*$. This implies that $\psi_f$ is uniquely determined by $(D|.)$. Since the infinity type is uniquely determined by the weight $k$ we have that there is at most one $(p^2,k)$-suitable compatible pair depending on whether the integral units compatibility is satisfied or not. As we have seen in the previous section, for $p\neq 3$ it is immediately satisfied. In the case $p=3$ one has that $1=w^3\equiv w\mod\mg$ and therefore since
$$w^{k-1}=\psi_\infty(w)=\psi_f^{-1}(w)=1$$
we get that $k\equiv1\mod3$. Furthermore $\Q(\sqrt{-3})$ has class number 1 so when $k\equiv1\mod3$ the unique $(9,k)$-suitable compatible pair is associated to a unique Hecke character that takes values in $\Q(\sqrt{-3})$. This will in turn give us a unique CM newform with coefficients in $\Q$, and therefore a unique CM newform Galois orbit. To sum it up:
$$\CM(9,k)=\Big\{\begin{matrix}1&k\equiv4\mod6\\0&\textrm{otherwise}\end{matrix}$$
For $p\neq3$ the pair $(\psi_f,\psi_\infty)$ is always compatible independently of the weight $k$. Let $h_K$ be the class number of $\Q(\sqrt{-p})$. Lemma \ref{lem:Hecke} then implies that there exist $h_K$ different Hecke characters and therefore, by Theorem  \ref{thm:Miyake}, we get that there exist $h_K$ many CM newforms. The way these CM newforms group into Galois orbits for every weight $k$ depends on $h_K$. To see this one has to go into more detail in the proof of Lemma \ref{lem:Hecke}. For simplicity we assume that the class group of $K$ is cyclic (which is true for all the $p$ in our computational range but certainly not true in general). In this case a Hecke character $\psi$ is completely determined by the value of $\psi(\mathfrak{a})$, which must satisfy the relation
$$\psi^{h_K}(\mathfrak{a})=\pi^{k-1}.$$
where $\mathfrak{a}$ is an ideal away from $\mg$ generating the class group and $\pi\in K$ is such that $\mathfrak{a}^{h_K}=(\pi)$ and $\psi_f(\pi)=1$. The possible values for $\psi(\mathfrak{a})$ are then $(\pi^{k-1})^{\frac{1}{h_K}}\zeta_{h_K}^i$, where $\zeta_{h_K}$ is a primitive $h_K$-th root of unity. The field $L_\psi$ where $\psi$ takes its values is then $\Q(\psi({\mathfrak{a}}))$ and the field $L_F$ of the corresponding CM newform is a totally real subfield of $\Q(\psi({\mathfrak{a}}))$. Their relative degree is $[L_\psi:L_f]=2$. Lets us give a more detailed description for some primes $p$:

\begin{itemize}
\item $p=23$. In this case $h_K=3$, so we get 3 CM newforms. Let $\psi$ be one of the 3 $(23^2,k)$-suitable Hecke characters. For $k\not\equiv4\mod6$ the degree of $L_\psi$ is $6$ and therefore the degree of $L_f$ is $3$. In this case all 3 CM newforms fall into a single Galois orbit. In the case $k\equiv4\mod6$ one of the Hecke characters has field $L_\psi=K$ and the other two have field $L_\psi=K(\zeta)$ which is of degree 4. In this case we get one CM form with rational coefficients that forms a Galois orbit by itself and two with coefficients in a real quadratic extension, that group together in another Galois orbit. To summarize:
$$\CM(23^2,k)=\Big\{\begin{matrix}2&k\equiv4\mod6\\1&\textrm{otherwise}\end{matrix}.$$
Clearly the same holds for any prime $p\equiv3\mod4$ for which $K$ has class number $h_K=3$.
\item $p=47$. In this case $h_K=5$, so we get 5 CM newforms and we have:
$$\CM(47^2,k)=\Big\{\begin{matrix}2&k\equiv6\mod10\\1&\textrm{otherwise}\end{matrix}.$$
Again the same formula holds for any prime $p\equiv3\mod4$ such that $h_K=5$.
\item Let $p\equiv3\mod4$ such that $h_k=\ell$, $\ell>2$ a rational prime. In this case there are $\ell$ CM newforms and we have:
$$\CM(p^2,k)=\Big\{\begin{matrix}2&k\equiv\ell+1\mod2\ell\\1&\textrm{otherwise}\end{matrix}.$$
\end{itemize}

\section{Observations}\label{sec:observations}

The following tables contain the suggested values for $\NCM(N)$ for $1\leq N \leq 100$ and, if it eventually seems to stabilize, the number of CM newform Galois orbits for the same level. These are observations that we got by looking at our computational data which we have included in Appendix \ref{app:Data}.
For every level we have computed $\NCM(N,k)$ and $\CM(N,k)$ for all weights up to at least $30$. For all $N$ in this range it is clear from the data we have collected that $\NCM(N,k)$ eventually stabilizes to some $\NCM(N)$. On the other hand $\CM(N,k)$ has a few cases where it does not stabilize (which we expected after the analysis of Section \ref{sec:CM}) but its behavior is clear even from the first few weights. We also provide the factorization of the level $N$ into prime powers which helps to observe the multiplicativity of $\NCM(N)$. To compute $\CM(N,k)$ and $\NCM(N,k)$ for level $N$ and weight $k$ we used standard commands of MAGMA \cite{MAGMA}, version 2.17 and higher:
\\
\\
\begin{tabular}{||r|r|r|r||r|r|r|r||r|r|r|r||r|r|r|r||}
\hline
N&fact&\begin{tiny}$\NCM$\end{tiny}&\begin{tiny}$\CM$\end{tiny}&N&fact&\begin{tiny}$\NCM$\end{tiny}&\begin{tiny}$\CM$\end{tiny}&N&fact&\begin{tiny}$\NCM$\end{tiny}&\begin{tiny}$\CM$\end{tiny}&N&fact&\begin{tiny}$\NCM$\end{tiny}&\begin{tiny}$\CM$\end{tiny}\\
\hline
1&$1$&1&0&26&$2^113^1$&4&0&51&$3^117^1$&4&0&76&$2^219^1$&2&0\\
2&$2^1$&2&0&27&$3^3$&4&1,0,1&52&$2^213^1$&2&0&77&$7^111^1$&4&0\\
3&$3^1$&2&0&28&$2^27^1$&2&0&53&$53^1$&2&0&78&$2^13^113^1$&8&0\\
4&$2^2$&1&0&29&$29^1$&2&0&54&$2^13^3$&8&0&79&$79^1$&2&0\\
5&$5^1$&2&0&30&$2^13^15^1$&8&0&55&$5^111^1$&4&0&80&$2^45^1$&12&0\\
6&$2^13^1$&4&0&31&$31^1$&2&0&56&$2^37^1$&4&0&81&$3^4$&5&0\\
7&$7^1$&2&0&32&$2^5$&4&1&57&$3^119^1$&4&0&82&$2^141^1$&4&0\\
8&$2^3$&2&0&33&$3^111^1$&4&0&58&$2^129^1$&4&0&83&$83^1$&2&0\\
9&$3^2$&4&0,1,0&34&$2^117^1$&4&0&59&$59^1$&2&0&84&$2^23^17^1$&4&0\\
10&$2^15^1$&4&0&35&$5^17^1$&4&0&60&$2^23^15^1$&4&0&85&$5^117^1$&4&0\\
11&$11^1$&2&0&36&$2^23^2$&4&1,0,1&61&$61^1$&2&0&86&$2^143^1$&4&0\\
12&$2^23^1$&2&0&37&$37^1$&2&0&62&$2^131^1$&4&0&87&$3^129^1$&4&0\\
13&$13^1$&2&0&38&$2^119^1$&4&0&63&$3^27^1$&8&0&88&$2^311^1$&4&0\\
14&$2^17^1$&4&0&39&$3^113^1$&4&0&64&$2^6$&16&1&89&$89^1$&2&0\\
15&$3^15^1$&4&0&40&$2^35^1$&4&0&65&$5^113^1$&4&0&90&$2^13^25^1$&16&0\\
16&$2^4$&6&0&41&$41^1$&2&0&66&$2^13^111^1$&8&0&91&$7^113^1$&4&0\\
17&$17^1$&2&0&42&$2^13^17^1$&8&0&67&$67^1$&2&0&92&$2^223^1$&2&0\\
18&$2^13^2$&8&1,0,1&43&$43^1$&2&0&68&$2^217^1$&2&0&93&$3^131^1$&4&0\\
19&$19^1$&2&0&44&$2^211^1$&2&0&69&$3^123^1$&4&0&94&$2^147^1$&4&0\\
20&$2^25^1$&2&0&45&$3^25^1$&8&0&70&$2^15^17^1$&8&0&95&$5^119^1$&4&0\\
21&$3^17^1$&4&0&46&$2^123^1$&4&0&71&$71^1$&2&0&96&$2^53^1$&8&0\\
22&$2^111^1$&4&0&47&$47^1$&2&0&72&$2^33^2$&8&0&97&$97^1$&2&0\\
23&$23^1$&2&0&48&$2^43^1$&12&0&73&$73^1$&2&0&98&$2^17^2$&14&0\\
24&$2^33^1$&4&0&49&$7^2$&7&1&74&$2^137^1$&4&0&99&$3^211^1$&8&0\\
25&$5^2$&6&0&50&$2^15^2$&12&0&75&$3^15^2$&12&0&100&$2^25^2$&6&0\\
\end{tabular}

For level $9$ we eventually have that $\NCM(9) = 4$, $\CM(9,k)=0$ for $k\equiv0,2\mod6$ and $\CM(9,k)=1$ for $k\equiv4\mod6$.

For level $27$ we eventually have that $\NCM(27) = 4$, $\CM(27,k)=1$ for $k\equiv0,2\mod6$ and $\CM(27,k)=0$ for $k\equiv4\mod6$.

For level $36$ we eventually have that $\NCM(36) = 4$, $\CM(36,k)=1$ for $k\equiv0,2\mod6$ and $\CM(36,k)=0$ for $k\equiv4\mod6$.

The following table contains the same information as the previous one but for levels $101\leq N \leq 200$. We used weights up to 30 with the exception of a few cases where we computed more weights to reach a more convincing point.\\
\\
\begin{tabular}{||r|r|r|r||r|r|r|r||r|r|r|r||r|r|r|r||}
\hline
N&fact&\begin{tiny}$\NCM$\end{tiny}&\begin{tiny}$\CM$\end{tiny}&N&fact&\begin{tiny}$\NCM$\end{tiny}&\begin{tiny}$\CM$\end{tiny}&N&fact&\begin{tiny}$\NCM$\end{tiny}&\begin{tiny}$\CM$\end{tiny}&N&fact&\begin{tiny}$\NCM$\end{tiny}&\begin{tiny}$\CM$\end{tiny}\\
\hline
101&$101^1$&2&0&126&$2^13^27^1$&16&0&151&$151^1$&2&0&176&$2^411^1$&12&0\\
102&$2^13^117^1$&8&0&127&$127^1$&2&0&152&$2^319^1$&4&0&177&$3^159^1$&4&0\\
103&$103^1$&2&0&128&$2^7$&8&0&153&$3^217^1$&8&0&178&$2^189^1$&4&0\\
104&$2^313^1$&4&0&129&$3^143^1$&4&0&154&$2^17^111^1$&8&0&179&$179^1$&2&0\\
105&$3^15^17^1$&8&0&130&$2^15^113^1$&8&0&155&$5^131^1$&4&0&180&$2^23^25^1$&8&0\\
106&$2^153^1$&4&0&131&$131^1$&2&0&156&$2^23^113^1$&4&0&181&$181^1$&2&0\\
107&$107^1$&2&0&132&$2^23^111^1$&4&0&157&$157^1$&2&0&182&$2^17^113^1$&8&0\\
108&$2^23^3$&4&1,2,1&133&$7^119^1$&4&0&158&$2^179^1$&4&0&183&$3^161^1$&4&0\\
109&$109^1$&2&0&134&$2^167^1$&4&0&159&$3^153^1$&4&0&184&$2^323^1$&4&0\\
110&$2^15^111^1$&8&0&135&$3^35^1$&8&0&160&$2^55^1$&8&0&185&$5^137^1$&4&0\\
111&$3^137^1$&4&0&136&$2^317^1$&4&0&161&$7^123^1$&4&0&186&$2^13^131^1$&8&0\\
112&$2^47^1$&12&0&137&$137^1$&2&0&162&$2^13^4$&10&0&187&$11^117^1$&4&0\\
113&$113^1$&2&0&138&$2^13^123^1$&8&0&163&$163^1$&2&0&188&$2^247^1$&2&0\\
114&$2^13^119^1$&8&0&139&$139^1$&2&0&164&$2^241^1$&2&0&189&$3^37^1$&8&0\\
115&$5^123^1$&4&0&140&$2^25^17^1$&4&0&165&$3^15^111^1$&8&0&190&$2^15^119^1$&8&0\\
116&$2^229^1$&2&0&141&$3^147^1$&4&0&166&$2^183^1$&4&0&191&$191^1$&2&0\\
117&$3^213^1$&8&0&142&$2^171^1$&4&0&167&$167^1$&2&0&192&$2^63^1$&32&0\\
118&$2^159^1$&4&0&143&$11^113^1$&4&0&168&$2^33^17^1$&8&0&193&$193^1$&2&0\\
119&$7^117^1$&4&0&144&$2^43^2$&24&1&169&$13^2$&9&0&194&$2^197^1$&4&0\\
120&$2^33^15^1$&8&0&145&$5^129^1$&4&0&170&$2^15^117^1$&8&0&195&$3^15^113^1$&8&0\\
121&$11^2$&9&1&146&$2^173^1$&4&0&171&$3^219^1$&8&0&196&$2^27^2$&7&0\\
122&$2^161^1$&4&0&147&$3^17^2$&14&0&172&$2^243^1$&2&0&197&$197^1$&2&0\\
123&$3^141^1$&4&0&148&$2^237^1$&2&0&173&$173^1$&2&0&198&$2^13^211^1$&16&0\\
124&$2^231^1$&2&0&149&$149^1$&2&0&174&$2^13^129^1$&8&0&199&$199^1$&2&0\\
125&$5^3$&4&0&150&$2^13^15^2$&24&0&175&$5^27^1$&12&0&200&$2^35^2$&12&0\\
\end{tabular}

For level $108$ we eventually have that $\NCM(108) = 4$, $\CM(108,k)=1$ for $k\equiv0,2\mod6$ and $\CM(108,k)=2$ for $k\equiv4\mod6$.

\section{Acknowledgments}
I would like to thank Gabor Wiese at Universit\'{e} du Luxembourg for providing access to his computing server, where most of the above computations were performed, and for many useful suggestions and comments. I would also like to thank Ulrich G\"{o}rtz at Institut f\"{u}r Experimentelle Mathematik of Universit\"{a}t Duisburg-Essen for providing access to his computing server, where some of the above computations were performed. Finally I would like to thank William Stein for the very useful suggestions he has made. This research was supported by the SPP 1489 priority program of the Deutsche Forschungsgemeinschaft.

\appendix

\section{Extended Computational Data}\label{app:Data}

In this section we provide tables with the data on which we based our observations mentioned in the previous sections. Every table below consists of rows corresponding to weights and columns corresponding to levels. For each level and weight we give a pair consisting of the number $\NCM(N,k)$ of non-CM newform Galois orbits in $S^{\textrm{new}}_k(N)$ and the number $\CM(N,k)$ of the CM Galois orbits in $S^{\textrm{new}}_k(N)$. We have done this systematically for all levels $1\leq N\leq 200$ and even weights $2\leq k\leq30$ at least. We also provide $\NCM(N,k)$ and $\CM(N,k)$ but for a much smaller weight range for specific levels $N > 200$ that are mostly prime powers. The goal was to understand the way $\NCM$ behaves on prime powers since, if one assumes that it is well defined and multiplicative, it would completely determine it. One can spot some patterns but the data range for these levels is too small to dismiss the possibility of a coincidence.\\
\\
\\
\begin{tabular}{r||r|r|r|r|r|r|r|r|r|r|}
&{\color{red}1}&{\color{red}2}&{\color{red}3}&{\color{red}4}&{\color{red}5}&{\color{red}6}&{\color{red}7}&{\color{red}8}&{\color{red}9}&{\color{red}10}\\
&{\color{red}$1$}&{\color{red}$2^1$}&{\color{red}$3^1$}&{\color{red}$2^2$}&{\color{red}$5^1$}&{\color{red}$2^13^1$}&{\color{red}$7^1$}&{\color{red}$2^3$}&{\color{red}$3^2$}&{\color{red}$2^15^1$}\\
{\color{green}2}&0 0&0 0&0 0&0 0&0 0&0 0&0 0&0 0&0 0&0 0\\
{\color{green}4}&0 0&0 0&0 0&0 0&1 0&1 0&1 0&1 0&0 1&1 0\\
{\color{green}6}&0 0&0 0&1 0&1 0&1 0&1 0&2 0&1 0&1 0&3 0\\
{\color{green}8}&0 0&1 0&1 0&0 0&2 0&1 0&2 0&2 0&2 0&1 0\\
{\color{green}10}&0 0&1 0&2 0&1 0&2 0&1 0&2 0&2 0&2 1&3 0\\
{\color{green}12}&1 0&0 0&1 0&1 0&2 0&3 0&2 0&2 0&3 0&4 0\\
{\color{green}14}&0 0&2 0&2 0&1 0&2 0&1 0&2 0&2 0&3 0&3 0\\
{\color{green}16}&1 0&1 0&2 0&1 0&2 0&3 0&2 0&3 0&4 1&4 0\\
{\color{green}18}&1 0&1 0&2 0&1 0&2 0&3 0&2 0&2 0&4 0&4 0\\
{\color{green}20}&1 0&2 0&2 0&1 0&2 0&3 0&2 0&2 0&4 0&4 0\\
{\color{green}22}&1 0&2 0&3 0&1 0&2 0&3 0&2 0&2 0&5 1&4 0\\
{\color{green}24}&1 0&1 0&2 0&1 0&2 0&4 0&2 0&2 0&4 0&4 0\\
{\color{green}26}&1 0&2 0&2 0&1 0&2 0&3 0&2 0&2 0&4 0&4 0\\
{\color{green}28}&1 0&2 0&2 0&1 0&2 0&4 0&2 0&2 0&4 1&4 0\\
{\color{green}30}&1 0&2 0&2 0&1 0&2 0&4 0&2 0&2 0&4 0&4 0\\
{\color{green}32}&1 0&2 0&2 0&1 0&2 0&4 0&2 0&2 0&4 0&4 0\\
{\color{green}34}&1 0&2 0&2 0&1 0&2 0&4 0&2 0&2 0&4 1&4 0\\
{\color{green}36}&1 0&2 0&2 0&1 0&2 0&4 0&2 0&2 0&4 0&4 0\\
{\color{green}38}&1 0&2 0&2 0&1 0&2 0&4 0&2 0&2 0&4 0&4 0\\
{\color{green}40}&1 0&2 0&2 0&1 0&2 0&4 0&2 0&2 0&4 1&4 0\\
\end{tabular}\\
\\
\begin{tabular}{r||r|r|r|r|r|r|r|r|r|r|}
&{\color{red}11}&{\color{red}12}&{\color{red}13}&{\color{red}14}&{\color{red}15}&{\color{red}16}&{\color{red}17}&{\color{red}18}&{\color{red}19}&{\color{red}20}\\
&{\color{red}$11^1$}&{\color{red}$2^23^1$}&{\color{red}$13^1$}&{\color{red}$2^17^1$}&{\color{red}$3^15^1$}&{\color{red}$2^4$}&{\color{red}$17^1$}&{\color{red}$2^13^2$}&{\color{red}$19^1$}&{\color{red}$2^25^1$}\\
{\color{green}2}&1 0&0 0&0 0&1 0&1 0&0 0&1 0&0 0&1 0&1 0\\
{\color{green}4}&1 0&1 0&2 0&2 0&2 0&1 0&2 0&1 0&2 0&1 0\\
{\color{green}6}&2 0&0 0&2 0&2 0&3 0&2 0&3 0&3 0&4 0&1 0\\
{\color{green}8}&2 0&2 0&3 0&3 0&3 0&3 0&3 0&2 0&2 0&2 0\\
{\color{green}10}&2 0&1 0&2 0&3 0&4 0&4 0&2 0&4 0&2 0&2 0\\
{\color{green}12}&2 0&2 0&2 0&4 0&4 0&4 0&2 0&5 0&2 0&2 0\\
{\color{green}14}&2 0&2 0&2 0&4 0&4 0&5 0&2 0&5 0&2 0&2 0\\
{\color{green}16}&2 0&2 0&2 0&4 0&4 0&6 0&2 0&6 0&2 0&2 0\\
{\color{green}18}&2 0&2 0&2 0&4 0&4 0&5 0&2 0&6 0&2 0&2 0\\
{\color{green}20}&2 0&2 0&2 0&4 0&4 0&6 0&2 0&7 0&2 0&2 0\\
{\color{green}22}&2 0&2 0&2 0&4 0&5 0&6 0&2 0&7 0&2 0&2 0\\
{\color{green}24}&2 0&2 0&2 0&4 0&4 0&5 0&2 0&7 0&2 0&2 0\\
{\color{green}26}&2 0&2 0&2 0&4 0&4 0&6 0&2 0&7 0&2 0&2 0\\
{\color{green}28}&2 0&2 0&2 0&4 0&4 0&6 0&2 0&8 0&2 0&2 0\\
{\color{green}30}&2 0&2 0&2 0&4 0&4 0&6 0&2 0&8 0&2 0&2 0\\
{\color{green}32}&2 0&2 0&2 0&4 0&4 0&6 0&2 0&8 0&2 0&2 0\\
{\color{green}34}&2 0&2 0&2 0&4 0&4 0&6 0&2 0&8 0&2 0&2 0\\
{\color{green}36}&2 0&2 0&2 0&4 0&4 0&6 0&2 0&8 0&2 0&2 0\\
{\color{green}38}&2 0&2 0&2 0&4 0&4 0&6 0&2 0&8 0&2 0&2 0\\
{\color{green}40}&2 0&2 0&2 0&4 0&4 0&6 0&2 0&8 0&2 0&2 0\\
\end{tabular}\\
\\
\begin{tabular}{r||r|r|r|r|r|r|r|r|r|r|}
&{\color{red}21}&{\color{red}22}&{\color{red}23}&{\color{red}24}&{\color{red}25}&{\color{red}26}&{\color{red}27}&{\color{red}28}&{\color{red}29}&{\color{red}30}\\
&{\color{red}$3^17^1$}&{\color{red}$2^111^1$}&{\color{red}$23^1$}&{\color{red}$2^33^1$}&{\color{red}$5^2$}&{\color{red}$2^113^1$}&{\color{red}$3^3$}&{\color{red}$2^27^1$}&{\color{red}$29^1$}&{\color{red}$2^13^15^1$}\\
{\color{green}2}&1 0&0 0&1 0&1 0&0 0&2 0&0 1&0 0&1 0&1 0\\
{\color{green}4}&3 0&3 0&2 0&1 0&3 0&3 0&3 0&2 0&2 0&2 0\\
{\color{green}6}&4 0&4 0&2 0&3 0&4 0&3 0&3 1&2 0&2 0&2 0\\
{\color{green}8}&4 0&4 0&2 0&3 0&5 0&5 0&4 1&2 0&2 0&6 0\\
{\color{green}10}&4 0&5 0&2 0&4 0&5 0&5 0&4 0&2 0&2 0&6 0\\
{\color{green}12}&5 0&4 0&2 0&4 0&6 0&4 0&4 1&2 0&2 0&6 0\\
{\color{green}14}&4 0&4 0&2 0&4 0&5 0&4 0&4 1&2 0&2 0&8 0\\
{\color{green}16}&4 0&4 0&2 0&4 0&6 0&4 0&4 0&2 0&2 0&8 0\\
{\color{green}18}&4 0&4 0&2 0&4 0&6 0&4 0&4 1&2 0&2 0&8 0\\
{\color{green}20}&4 0&5 0&2 0&4 0&6 0&4 0&4 1&2 0&2 0&8 0\\
{\color{green}22}&4 0&4 0&2 0&4 0&6 0&4 0&4 0&2 0&2 0&8 0\\
{\color{green}24}&4 0&4 0&2 0&4 0&6 0&4 0&4 1&2 0&2 0&8 0\\
{\color{green}26}&4 0&4 0&2 0&4 0&6 0&4 0&4 1&2 0&2 0&8 0\\
{\color{green}28}&4 0&4 0&2 0&4 0&6 0&4 0&4 0&2 0&2 0&8 0\\
{\color{green}30}&4 0&4 0&2 0&4 0&6 0&4 0&4 1&2 0&2 0&8 0\\
{\color{green}32}&4 0&4 0&2 0&4 0&6 0&4 0&4 1&2 0&2 0&8 0\\
{\color{green}34}&4 0&4 0&2 0&4 0&6 0&4 0&4 0&2 0&2 0&8 0\\
{\color{green}36}&4 0&4 0&2 0&4 0&6 0&4 0&4 1&2 0&2 0&8 0\\
{\color{green}38}&4 0&4 0&2 0&4 0&6 0&4 0&4 1&2 0&2 0&8 0\\
{\color{green}40}&4 0&4 0&2 0&4 0&6 0&4 0&4 0&2 0&2 0&8 0\\
\end{tabular}\\
\\
\begin{tabular}{r||r|r|r|r|r|r|r|r|r|r|}
&{\color{red}31}&{\color{red}32}&{\color{red}33}&{\color{red}34}&{\color{red}35}&{\color{red}36}&{\color{red}37}&{\color{red}38}&{\color{red}39}&{\color{red}40}\\
&{\color{red}$31^1$}&{\color{red}$2^5$}&{\color{red}$3^111^1$}&{\color{red}$2^117^1$}&{\color{red}$5^17^1$}&{\color{red}$2^23^2$}&{\color{red}$37^1$}&{\color{red}$2^119^1$}&{\color{red}$3^113^1$}&{\color{red}$2^35^1$}\\
{\color{green}2}&1 0&0 1&1 0&1 0&2 0&0 1&2 0&2 0&2 0&1 0\\
{\color{green}4}&2 0&2 1&4 0&3 0&3 0&1 0&2 0&3 0&3 0&3 0\\
{\color{green}6}&2 0&3 1&5 0&4 0&4 0&1 1&2 0&4 0&4 0&4 0\\
{\color{green}8}&2 0&3 1&5 0&5 0&4 0&2 1&2 0&5 0&4 0&4 0\\
{\color{green}10}&2 0&4 1&4 0&4 0&5 0&3 0&2 0&5 0&4 0&4 0\\
{\color{green}12}&2 0&4 1&4 0&4 0&4 0&3 1&2 0&4 0&4 0&5 0\\
{\color{green}14}&2 0&4 1&5 0&4 0&4 0&4 1&2 0&4 0&4 0&4 0\\
{\color{green}16}&2 0&4 1&4 0&4 0&4 0&4 0&2 0&4 0&4 0&4 0\\
{\color{green}18}&2 0&5 1&4 0&4 0&4 0&4 1&2 0&4 0&4 0&4 0\\
{\color{green}20}&2 0&4 1&4 0&4 0&4 0&4 1&2 0&4 0&4 0&4 0\\
{\color{green}22}&2 0&4 1&4 0&4 0&4 0&4 0&2 0&4 0&4 0&4 0\\
{\color{green}24}&2 0&4 1&4 0&4 0&4 0&4 1&2 0&4 0&4 0&4 0\\
{\color{green}26}&2 0&4 1&4 0&4 0&4 0&4 1&2 0&4 0&4 0&4 0\\
{\color{green}28}&2 0&4 1&4 0&4 0&4 0&4 0&2 0&4 0&4 0&4 0\\
{\color{green}30}&2 0&4 1&4 0&4 0&4 0&4 1&2 0&4 0&4 0&4 0\\
{\color{green}32}&2 0&4 1&4 0&4 0&4 0&4 1&2 0&4 0&4 0&4 0\\
{\color{green}34}&2 0&4 1&4 0&4 0&4 0&4 0&2 0&4 0&4 0&4 0\\
{\color{green}36}&2 0&4 1&4 0&4 0&4 0&4 1&2 0&4 0&4 0&4 0\\
{\color{green}38}&2 0&4 1&4 0&4 0&4 0&4 1&2 0&4 0&4 0&4 0\\
{\color{green}40}&2 0&4 1&4 0&4 0&4 0&4 0&2 0&4 0&4 0&4 0\\
\end{tabular}\\
\\
\begin{tabular}{r||r|r|r|r|r|r|r|r|r|r|}
&{\color{red}41}&{\color{red}42}&{\color{red}43}&{\color{red}44}&{\color{red}45}&{\color{red}46}&{\color{red}47}&{\color{red}48}&{\color{red}49}&{\color{red}50}\\
&{\color{red}$41^1$}&{\color{red}$2^13^17^1$}&{\color{red}$43^1$}&{\color{red}$2^211^1$}&{\color{red}$3^25^1$}&{\color{red}$2^123^1$}&{\color{red}$47^1$}&{\color{red}$2^43^1$}&{\color{red}$7^2$}&{\color{red}$2^15^2$}\\
{\color{green}2}&1 0&1 0&2 0&1 0&1 0&1 0&1 0&1 0&0 1&2 0\\
{\color{green}4}&2 0&2 0&2 0&2 0&5 0&4 0&2 0&3 0&4 1&5 0\\
{\color{green}6}&2 0&6 0&2 0&2 0&6 0&4 0&2 0&5 0&6 1&7 0\\
{\color{green}8}&2 0&6 0&2 0&2 0&9 0&4 0&2 0&7 0&6 1&10 0\\
{\color{green}10}&2 0&8 0&2 0&2 0&8 0&5 0&2 0&8 0&6 1&10 0\\
{\color{green}12}&2 0&8 0&2 0&2 0&8 0&4 0&2 0&10 0&7 1&10 0\\
{\color{green}14}&2 0&9 0&2 0&2 0&8 0&4 0&2 0&9 0&6 1&11 0\\
{\color{green}16}&2 0&9 0&2 0&2 0&8 0&4 0&2 0&11 0&7 1&11 0\\
{\color{green}18}&2 0&8 0&2 0&2 0&8 0&4 0&2 0&11 0&7 1&11 0\\
{\color{green}20}&2 0&8 0&2 0&2 0&8 0&4 0&2 0&11 0&7 1&12 0\\
{\color{green}22}&2 0&8 0&2 0&2 0&9 0&4 0&2 0&12 0&7 1&12 0\\
{\color{green}24}&2 0&8 0&2 0&2 0&8 0&4 0&2 0&12 0&7 1&11 0\\
{\color{green}26}&2 0&8 0&2 0&2 0&8 0&4 0&2 0&11 0&7 1&12 0\\
{\color{green}28}&2 0&8 0&2 0&2 0&8 0&4 0&2 0&12 0&7 1&12 0\\
{\color{green}30}&2 0&8 0&2 0&2 0&8 0&4 0&2 0&12 0&7 1&12 0\\
{\color{green}32}&2 0&8 0&2 0&2 0&8 0&4 0&2 0&12 0&7 1&12 0\\
{\color{green}34}&2 0&8 0&2 0&2 0&8 0&4 0&2 0&12 0&7 1&12 0\\
{\color{green}36}&2 0&8 0&2 0&2 0&8 0&4 0&2 0&12 0&7 1&12 0\\
{\color{green}38}&2 0&8 0&2 0&2 0&8 0&4 0&2 0&12 0&7 1&12 0\\
{\color{green}40}&2 0&8 0&2 0&2 0&8 0&4 0&2 0&12 0&7 1&12 0\\
\end{tabular}\\
\\
\begin{tabular}{r||r|r|r|r|r|r|r|r|r|r|}
&{\color{red}51}&{\color{red}52}&{\color{red}53}&{\color{red}54}&{\color{red}55}&{\color{red}56}&{\color{red}57}&{\color{red}58}&{\color{red}59}&{\color{red}60}\\
&{\color{red}$3^117^1$}&{\color{red}$2^213^1$}&{\color{red}$53^1$}&{\color{red}$2^13^3$}&{\color{red}$5^111^1$}&{\color{red}$2^37^1$}&{\color{red}$3^119^1$}&{\color{red}$2^129^1$}&{\color{red}$59^1$}&{\color{red}$2^23^15^1$}\\
{\color{green}2}&2 0&1 0&2 0&2 0&2 0&2 0&3 0&2 0&1 0&0 0\\
{\color{green}4}&5 0&2 0&3 0&4 0&4 0&3 0&4 0&4 0&3 0&2 0\\
{\color{green}6}&4 0&3 0&2 0&6 0&4 0&5 0&6 0&4 0&2 0&4 0\\
{\color{green}8}&5 0&3 0&2 0&8 0&4 0&5 0&4 0&4 0&2 0&4 0\\
{\color{green}10}&4 0&2 0&2 0&8 0&4 0&4 0&4 0&4 0&2 0&4 0\\
{\color{green}12}&4 0&2 0&2 0&8 0&4 0&4 0&4 0&4 0&2 0&4 0\\
{\color{green}14}&4 0&2 0&2 0&8 0&4 0&4 0&4 0&4 0&2 0&5 0\\
{\color{green}16}&4 0&2 0&2 0&8 0&4 0&4 0&4 0&4 0&2 0&4 0\\
{\color{green}18}&4 0&2 0&2 0&8 0&4 0&4 0&4 0&4 0&2 0&4 0\\
{\color{green}20}&4 0&2 0&2 0&8 0&4 0&4 0&4 0&4 0&2 0&4 0\\
{\color{green}22}&4 0&2 0&2 0&8 0&4 0&4 0&4 0&4 0&2 0&4 0\\
{\color{green}24}&4 0&2 0&2 0&8 0&4 0&4 0&4 0&4 0&2 0&4 0\\
{\color{green}26}&4 0&2 0&2 0&8 0&4 0&4 0&4 0&4 0&2 0&4 0\\
{\color{green}28}&4 0&2 0&2 0&8 0&4 0&4 0&4 0&4 0&2 0&4 0\\
{\color{green}30}&4 0&2 0&2 0&8 0&4 0&4 0&4 0&4 0&2 0&4 0\\
{\color{green}32}&4 0&2 0&2 0&8 0&4 0&4 0&4 0&4 0&2 0&4 0\\
{\color{green}34}&4 0&2 0&2 0&8 0&4 0&4 0&4 0&4 0&2 0&4 0\\
{\color{green}36}&4 0&2 0&2 0&8 0&4 0&4 0&4 0&4 0&2 0&-\ -\\
{\color{green}38}&4 0&2 0&2 0&8 0&4 0&4 0&4 0&4 0&2 0&-\ -\\
{\color{green}40}&4 0&2 0&2 0&8 0&4 0&4 0&4 0&4 0&2 0&-\ -\\
\end{tabular}\\
\\
\begin{tabular}{r||r|r|r|r|r|r|r|r|r|r|}
&{\color{red}61}&{\color{red}62}&{\color{red}63}&{\color{red}64}&{\color{red}65}&{\color{red}66}&{\color{red}67}&{\color{red}68}&{\color{red}69}&{\color{red}70}\\
&{\color{red}$61^1$}&{\color{red}$2^131^1$}&{\color{red}$3^27^1$}&{\color{red}$2^6$}&{\color{red}$5^113^1$}&{\color{red}$2^13^111^1$}&{\color{red}$67^1$}&{\color{red}$2^217^1$}&{\color{red}$3^123^1$}&{\color{red}$2^15^17^1$}\\
{\color{green}2}&2 0&2 0&2 0&0 1&3 0&3 0&3 0&1 0&2 0&1 0\\
{\color{green}4}&2 0&4 0&5 0&4 1&5 0&3 0&2 0&2 0&4 0&6 0\\
{\color{green}6}&2 0&4 0&8 0&7 1&5 0&6 0&2 0&2 0&5 0&8 0\\
{\color{green}8}&2 0&4 0&8 0&9 1&4 0&7 0&2 0&2 0&4 0&8 0\\
{\color{green}10}&2 0&5 0&8 0&12 1&4 0&8 0&2 0&2 0&4 0&9 0\\
{\color{green}12}&2 0&4 0&9 0&12 1&4 0&8 0&2 0&2 0&4 0&8 0\\
{\color{green}14}&2 0&4 0&8 0&14 1&4 0&8 0&2 0&2 0&4 0&8 0\\
{\color{green}16}&2 0&4 0&8 0&16 1&4 0&8 0&2 0&2 0&4 0&8 0\\
{\color{green}18}&2 0&4 0&8 0&15 1&4 0&8 0&2 0&2 0&4 0&8 0\\
{\color{green}20}&2 0&4 0&8 0&16 1&4 0&8 0&2 0&2 0&4 0&8 0\\
{\color{green}22}&2 0&4 0&8 0&16 1&4 0&8 0&2 0&2 0&4 0&8 0\\
{\color{green}24}&2 0&4 0&8 0&14 1&4 0&8 0&2 0&2 0&4 0&8 0\\
{\color{green}26}&2 0&4 0&8 0&16 1&4 0&8 0&2 0&2 0&4 0&8 0\\
{\color{green}28}&2 0&4 0&8 0&16 1&4 0&8 0&2 0&2 0&4 0&8 0\\
{\color{green}30}&2 0&4 0&8 0&16 1&4 0&8 0&2 0&2 0&4 0&8 0\\
{\color{green}32}&2 0&4 0&8 0&16 1&4 0&8 0&2 0&2 0&4 0&8 0\\
{\color{green}34}&2 0&-\ -&-\ -&-\ -&-\ -&-\ -&-\ -&-\ -&-\ -&-\ -\\
{\color{green}36}&-\ -&-\ -&-\ -&-\ -&-\ -&-\ -&-\ -&-\ -&-\ -&-\ -\\
{\color{green}38}&-\ -&-\ -&-\ -&-\ -&-\ -&-\ -&-\ -&-\ -&-\ -&-\ -\\
{\color{green}40}&-\ -&-\ -&-\ -&-\ -&-\ -&-\ -&-\ -&-\ -&-\ -&-\ -\\
\end{tabular}\\
\\
\begin{tabular}{r||r|r|r|r|r|r|r|r|r|r|}
&{\color{red}71}&{\color{red}72}&{\color{red}73}&{\color{red}74}&{\color{red}75}&{\color{red}76}&{\color{red}77}&{\color{red}78}&{\color{red}79}&{\color{red}80}\\
&{\color{red}$71^1$}&{\color{red}$2^33^2$}&{\color{red}$73^1$}&{\color{red}$2^137^1$}&{\color{red}$3^15^2$}&{\color{red}$2^219^1$}&{\color{red}$7^111^1$}&{\color{red}$2^13^113^1$}&{\color{red}$79^1$}&{\color{red}$2^45^1$}\\
{\color{green}2}&2 0&1 0&3 0&2 0&3 0&1 0&4 0&1 0&2 0&2 0\\
{\color{green}4}&3 0&4 0&3 0&4 0&6 0&2 0&5 0&6 0&3 0&6 0\\
{\color{green}6}&2 0&6 0&2 0&5 0&10 0&2 0&5 0&8 0&2 0&9 0\\
{\color{green}8}&2 0&7 0&2 0&4 0&10 0&2 0&4 0&8 0&2 0&9 0\\
{\color{green}10}&2 0&8 0&2 0&4 0&12 0&2 0&4 0&8 0&2 0&11 0\\
{\color{green}12}&2 0&8 0&2 0&4 0&11 0&2 0&4 0&8 0&2 0&13 0\\
{\color{green}14}&2 0&8 0&2 0&4 0&12 0&2 0&4 0&8 0&2 0&11 0\\
{\color{green}16}&2 0&9 0&2 0&4 0&12 0&2 0&4 0&8 0&2 0&12 0\\
{\color{green}18}&2 0&8 0&2 0&4 0&12 0&2 0&4 0&8 0&2 0&12 0\\
{\color{green}20}&2 0&8 0&2 0&4 0&12 0&2 0&4 0&8 0&2 0&12 0\\
{\color{green}22}&2 0&8 0&2 0&4 0&14 0&2 0&4 0&8 0&2 0&12 0\\
{\color{green}24}&2 0&8 0&2 0&4 0&12 0&2 0&4 0&8 0&2 0&12 0\\
{\color{green}26}&2 0&8 0&2 0&4 0&12 0&2 0&4 0&8 0&2 0&12 0\\
{\color{green}28}&2 0&8 0&2 0&4 0&12 0&2 0&4 0&8 0&2 0&12 0\\
{\color{green}30}&2 0&8 0&2 0&4 0&12 0&2 0&4 0&8 0&2 0&12 0\\
{\color{green}32}&2 0&8 0&2 0&4 0&12 0&2 0&4 0&8 0&2 0&-\ -\\
{\color{green}34}&-\ -&-\ -&-\ -&-\ -&-\ -&-\ -&-\ -&-\ -&-\ -&-\ -\\
{\color{green}36}&-\ -&-\ -&-\ -&-\ -&-\ -&-\ -&-\ -&-\ -&-\ -&-\ -\\
{\color{green}38}&-\ -&-\ -&-\ -&-\ -&-\ -&-\ -&-\ -&-\ -&-\ -&-\ -\\
{\color{green}40}&-\ -&-\ -&-\ -&-\ -&-\ -&-\ -&-\ -&-\ -&-\ -&-\ -\\
\end{tabular}\\
\\
\begin{tabular}{r||r|r|r|r|r|r|r|r|r|r|}
&{\color{red}81}&{\color{red}82}&{\color{red}83}&{\color{red}84}&{\color{red}85}&{\color{red}86}&{\color{red}87}&{\color{red}88}&{\color{red}89}&{\color{red}90}\\
&{\color{red}$3^4$}&{\color{red}$2^141^1$}&{\color{red}$83^1$}&{\color{red}$2^23^17^1$}&{\color{red}$5^117^1$}&{\color{red}$2^143^1$}&{\color{red}$3^129^1$}&{\color{red}$2^311^1$}&{\color{red}$89^1$}&{\color{red}$2^13^25^1$}\\
{\color{green}2}&1 0&2 0&2 0&2 0&3 0&2 0&2 0&2 0&3 0&3 0\\
{\color{green}4}&5 0&5 0&2 0&2 0&7 0&5 0&4 0&4 0&4 0&5 0\\
{\color{green}6}&5 0&4 0&2 0&4 0&4 0&4 0&4 0&4 0&2 0&7 0\\
{\color{green}8}&5 0&4 0&2 0&4 0&4 0&4 0&4 0&4 0&2 0&11 0\\
{\color{green}10}&5 0&4 0&2 0&4 0&4 0&4 0&4 0&4 0&2 0&13 0\\
{\color{green}12}&5 0&4 0&2 0&5 0&4 0&4 0&4 0&4 0&2 0&14 0\\
{\color{green}14}&5 0&4 0&2 0&4 0&4 0&4 0&4 0&4 0&2 0&15 0\\
{\color{green}16}&5 0&4 0&2 0&4 0&4 0&4 0&4 0&4 0&2 0&16 0\\
{\color{green}18}&5 0&4 0&2 0&4 0&4 0&4 0&4 0&4 0&2 0&16 0\\
{\color{green}20}&5 0&4 0&2 0&4 0&4 0&4 0&4 0&4 0&2 0&16 0\\
{\color{green}22}&5 0&4 0&2 0&4 0&4 0&4 0&4 0&4 0&2 0&16 0\\
{\color{green}24}&5 0&4 0&2 0&4 0&4 0&4 0&4 0&4 0&2 0&16 0\\
{\color{green}26}&5 0&4 0&2 0&4 0&4 0&4 0&4 0&4 0&2 0&16 0\\
{\color{green}28}&5 0&4 0&2 0&4 0&4 0&4 0&4 0&4 0&2 0&16 0\\
{\color{green}30}&5 0&4 0&2 0&4 0&4 0&4 0&4 0&4 0&2 0&16 0\\
{\color{green}32}&5 0&4 0&2 0&4 0&4 0&4 0&4 0&4 0&2 0&16 0\\
{\color{green}34}&-\ -&-\ -&-\ -&-\ -&-\ -&-\ -&-\ -&-\ -&-\ -&-\ -\\
{\color{green}36}&-\ -&-\ -&-\ -&-\ -&-\ -&-\ -&-\ -&-\ -&-\ -&-\ -\\
{\color{green}38}&-\ -&-\ -&-\ -&-\ -&-\ -&-\ -&-\ -&-\ -&-\ -&-\ -\\
{\color{green}40}&-\ -&-\ -&-\ -&-\ -&-\ -&-\ -&-\ -&-\ -&-\ -&-\ -\\
\end{tabular}\\
\\
\begin{tabular}{r||r|r|r|r|r|r|r|r|r|r|}
&{\color{red}91}&{\color{red}92}&{\color{red}93}&{\color{red}94}&{\color{red}95}&{\color{red}96}&{\color{red}97}&{\color{red}98}&{\color{red}99}&{\color{red}100}\\
&{\color{red}$7^113^1$}&{\color{red}$2^223^1$}&{\color{red}$3^131^1$}&{\color{red}$2^147^1$}&{\color{red}$5^119^1$}&{\color{red}$2^53^1$}&{\color{red}$97^1$}&{\color{red}$2^17^2$}&{\color{red}$3^211^1$}&{\color{red}$2^25^2$}\\
{\color{green}2}&4 0&2 0&2 0&2 0&2 0&2 0&2 0&2 0&4 0&1 0\\
{\color{green}4}&4 0&2 0&6 0&4 0&7 0&6 0&2 0&8 0&7 0&4 0\\
{\color{green}6}&4 0&2 0&4 0&4 0&5 0&8 0&2 0&9 0&9 0&5 0\\
{\color{green}8}&5 0&2 0&4 0&4 0&4 0&8 0&2 0&12 0&9 0&5 0\\
{\color{green}10}&4 0&2 0&4 0&4 0&4 0&8 0&2 0&12 0&8 0&6 0\\
{\color{green}12}&4 0&2 0&4 0&4 0&4 0&10 0&2 0&13 0&8 0&6 0\\
{\color{green}14}&4 0&2 0&4 0&4 0&4 0&8 0&2 0&14 0&9 0&6 0\\
{\color{green}16}&4 0&2 0&4 0&4 0&4 0&8 0&2 0&13 0&8 0&6 0\\
{\color{green}18}&4 0&2 0&4 0&4 0&4 0&8 0&2 0&13 0&8 0&6 0\\
{\color{green}20}&4 0&2 0&4 0&4 0&4 0&8 0&2 0&14 0&8 0&6 0\\
{\color{green}22}&4 0&2 0&4 0&4 0&4 0&8 0&2 0&14 0&8 0&6 0\\
{\color{green}24}&4 0&2 0&4 0&4 0&4 0&8 0&2 0&13 0&8 0&6 0\\
{\color{green}26}&4 0&2 0&4 0&4 0&4 0&8 0&2 0&14 0&8 0&6 0\\
{\color{green}28}&4 0&2 0&4 0&4 0&4 0&8 0&2 0&14 0&8 0&6 0\\
{\color{green}30}&4 0&2 0&4 0&4 0&4 0&8 0&2 0&14 0&8 0&6 0\\
{\color{green}32}&4 0&2 0&4 0&4 0&4 0&8 0&2 0&14 0&-\ -&-\ -\\
{\color{green}34}&-\ -&-\ -&-\ -&-\ -&-\ -&-\ -&-\ -&-\ -&-\ -&-\ -\\
{\color{green}36}&-\ -&-\ -&-\ -&-\ -&-\ -&-\ -&-\ -&-\ -&-\ -&-\ -\\
{\color{green}38}&-\ -&-\ -&-\ -&-\ -&-\ -&-\ -&-\ -&-\ -&-\ -&-\ -\\
{\color{green}40}&-\ -&-\ -&-\ -&-\ -&-\ -&-\ -&-\ -&-\ -&-\ -&-\ -\\
\end{tabular}\\
\\
\begin{tabular}{r||r|r|r|r|r|r|r|r|r|r|}
&{\color{red}101}&{\color{red}102}&{\color{red}103}&{\color{red}104}&{\color{red}105}&{\color{red}106}&{\color{red}107}&{\color{red}108}&{\color{red}109}&{\color{red}110}\\
&{\color{red}$101^1$}&{\color{red}$2^13^117^1$}&{\color{red}$103^1$}&{\color{red}$2^313^1$}&{\color{red}$3^15^17^1$}&{\color{red}$2^153^1$}&{\color{red}$107^1$}&{\color{red}$2^23^3$}&{\color{red}$109^1$}&{\color{red}$2^15^111^1$}\\
{\color{green}2}&2 0&3 0&2 0&2 0&2 0&4 0&2 0&0 1&3 0&4 0\\
{\color{green}4}&2 0&6 0&2 0&5 0&7 0&4 0&2 0&2 2&2 0&9 0\\
{\color{green}6}&2 0&8 0&2 0&5 0&8 0&4 0&2 0&3 1&2 0&8 0\\
{\color{green}8}&2 0&9 0&2 0&4 0&9 0&4 0&2 0&4 1&2 0&8 0\\
{\color{green}10}&2 0&9 0&2 0&4 0&8 0&4 0&2 0&4 2&2 0&8 0\\
{\color{green}12}&2 0&8 0&2 0&4 0&8 0&4 0&2 0&4 1&2 0&8 0\\
{\color{green}14}&2 0&8 0&2 0&4 0&8 0&4 0&2 0&4 1&2 0&8 0\\
{\color{green}16}&2 0&8 0&2 0&4 0&8 0&4 0&2 0&4 2&2 0&8 0\\
{\color{green}18}&2 0&8 0&2 0&4 0&8 0&4 0&2 0&4 1&2 0&8 0\\
{\color{green}20}&2 0&8 0&2 0&4 0&8 0&4 0&2 0&4 1&2 0&8 0\\
{\color{green}22}&2 0&8 0&2 0&4 0&8 0&4 0&2 0&4 2&2 0&8 0\\
{\color{green}24}&2 0&8 0&2 0&4 0&8 0&4 0&2 0&4 1&2 0&8 0\\
{\color{green}26}&2 0&8 0&2 0&4 0&8 0&4 0&2 0&4 1&2 0&8 0\\
{\color{green}28}&2 0&8 0&2 0&4 0&8 0&4 0&2 0&4 2&2 0&8 0\\
{\color{green}30}&2 0&8 0&2 0&4 0&8 0&4 0&2 0&4 1&2 0&8 0\\
{\color{green}32}&2 0&-\ -&-\ -&-\ -&-\ -&-\ -&-\ -&-\ -&-\ -&-\ -\\
{\color{green}34}&-\ -&-\ -&-\ -&-\ -&-\ -&-\ -&-\ -&-\ -&-\ -&-\ -\\
{\color{green}36}&-\ -&-\ -&-\ -&-\ -&-\ -&-\ -&-\ -&-\ -&-\ -&-\ -\\
{\color{green}38}&-\ -&-\ -&-\ -&-\ -&-\ -&-\ -&-\ -&-\ -&-\ -&-\ -\\
{\color{green}40}&-\ -&-\ -&-\ -&-\ -&-\ -&-\ -&-\ -&-\ -&-\ -&-\ -\\
\end{tabular}\\
\\
\begin{tabular}{r||r|r|r|r|r|r|r|r|r|r|}
&{\color{red}111}&{\color{red}112}&{\color{red}113}&{\color{red}114}&{\color{red}115}&{\color{red}116}&{\color{red}117}&{\color{red}118}&{\color{red}119}&{\color{red}120}\\
&{\color{red}$3^137^1$}&{\color{red}$2^47^1$}&{\color{red}$113^1$}&{\color{red}$2^13^119^1$}&{\color{red}$5^123^1$}&{\color{red}$2^229^1$}&{\color{red}$3^213^1$}&{\color{red}$2^159^1$}&{\color{red}$7^117^1$}&{\color{red}$2^33^15^1$}\\
{\color{green}2}&2 0&3 0&4 0&3 0&3 0&3 0&3 0&4 0&2 0&2 0\\
{\color{green}4}&6 0&8 0&2 0&6 0&6 0&3 0&7 0&6 0&5 0&6 0\\
{\color{green}6}&4 0&11 0&2 0&9 0&5 0&4 0&8 0&4 0&4 0&8 0\\
{\color{green}8}&4 0&12 0&2 0&10 0&4 0&2 0&9 0&4 0&4 0&8 0\\
{\color{green}10}&4 0&11 0&2 0&8 0&4 0&2 0&8 0&4 0&4 0&8 0\\
{\color{green}12}&4 0&12 0&2 0&9 0&4 0&2 0&8 0&4 0&4 0&8 0\\
{\color{green}14}&4 0&12 0&2 0&8 0&4 0&2 0&8 0&4 0&4 0&8 0\\
{\color{green}16}&4 0&12 0&2 0&8 0&4 0&2 0&8 0&4 0&4 0&8 0\\
{\color{green}18}&4 0&12 0&2 0&8 0&4 0&2 0&8 0&4 0&4 0&8 0\\
{\color{green}20}&4 0&12 0&2 0&8 0&4 0&2 0&8 0&4 0&4 0&8 0\\
{\color{green}22}&4 0&12 0&2 0&8 0&4 0&2 0&8 0&4 0&4 0&8 0\\
{\color{green}24}&4 0&12 0&2 0&8 0&4 0&2 0&8 0&4 0&4 0&8 0\\
{\color{green}26}&4 0&12 0&2 0&8 0&4 0&2 0&8 0&4 0&4 0&8 0\\
{\color{green}28}&4 0&12 0&2 0&8 0&4 0&2 0&8 0&4 0&4 0&8 0\\
{\color{green}30}&4 0&12 0&2 0&8 0&4 0&2 0&8 0&4 0&4 0&8 0\\
{\color{green}32}&-\ -&-\ -&-\ -&-\ -&-\ -&-\ -&-\ -&-\ -&-\ -&-\ -\\
{\color{green}34}&-\ -&-\ -&-\ -&-\ -&-\ -&-\ -&-\ -&-\ -&-\ -&-\ -\\
{\color{green}36}&-\ -&-\ -&-\ -&-\ -&-\ -&-\ -&-\ -&-\ -&-\ -&-\ -\\
{\color{green}38}&-\ -&-\ -&-\ -&-\ -&-\ -&-\ -&-\ -&-\ -&-\ -&-\ -\\
{\color{green}40}&-\ -&-\ -&-\ -&-\ -&-\ -&-\ -&-\ -&-\ -&-\ -&-\ -\\
\end{tabular}\\
\\
\begin{tabular}{r||r|r|r|r|r|r|r|r|r|r|}
&{\color{red}121}&{\color{red}122}&{\color{red}123}&{\color{red}124}&{\color{red}125}&{\color{red}126}&{\color{red}127}&{\color{red}128}&{\color{red}129}&{\color{red}130}\\
&{\color{red}$11^2$}&{\color{red}$2^161^1$}&{\color{red}$3^141^1$}&{\color{red}$2^231^1$}&{\color{red}$5^3$}&{\color{red}$2^13^27^1$}&{\color{red}$127^1$}&{\color{red}$2^7$}&{\color{red}$3^143^1$}&{\color{red}$2^15^113^1$}\\
{\color{green}2}&3 1&3 0&4 0&2 0&3 0&2 0&2 0&4 0&4 0&3 0\\
{\color{green}4}&7 1&4 0&4 0&2 0&4 0&8 0&3 0&8 0&6 0&7 0\\
{\color{green}6}&8 1&5 0&4 0&2 0&4 0&12 0&2 0&12 0&4 0&8 0\\
{\color{green}8}&8 1&4 0&4 0&2 0&4 0&13 0&2 0&8 0&4 0&10 0\\
{\color{green}10}&8 1&4 0&4 0&2 0&4 0&15 0&2 0&12 0&4 0&8 0\\
{\color{green}12}&9 1&4 0&4 0&2 0&4 0&16 0&2 0&8 0&4 0&8 0\\
{\color{green}14}&8 1&4 0&4 0&2 0&4 0&17 0&2 0&8 0&4 0&8 0\\
{\color{green}16}&9 1&4 0&4 0&2 0&4 0&17 0&2 0&8 0&4 0&8 0\\
{\color{green}18}&9 1&4 0&4 0&2 0&4 0&16 0&2 0&8 0&4 0&8 0\\
{\color{green}20}&9 1&4 0&4 0&2 0&4 0&16 0&2 0&8 0&4 0&8 0\\
{\color{green}22}&9 1&4 0&4 0&2 0&4 0&16 0&2 0&8 0&4 0&8 0\\
{\color{green}24}&9 1&4 0&4 0&2 0&4 0&16 0&2 0&8 0&4 0&8 0\\
{\color{green}26}&9 1&4 0&4 0&2 0&4 0&16 0&2 0&8 0&4 0&8 0\\
{\color{green}28}&9 1&4 0&4 0&2 0&4 0&16 0&2 0&8 0&4 0&8 0\\
{\color{green}30}&9 1&4 0&4 0&2 0&4 0&16 0&2 0&8 0&4 0&8 0\\
{\color{green}32}&-\ -&-\ -&-\ -&-\ -&-\ -&-\ -&-\ -&-\ -&-\ -&-\ -\\
{\color{green}34}&-\ -&-\ -&-\ -&-\ -&-\ -&-\ -&-\ -&-\ -&-\ -&-\ -\\
{\color{green}36}&-\ -&-\ -&-\ -&-\ -&-\ -&-\ -&-\ -&-\ -&-\ -&-\ -\\
{\color{green}38}&-\ -&-\ -&-\ -&-\ -&-\ -&-\ -&-\ -&-\ -&-\ -&-\ -\\
{\color{green}40}&-\ -&-\ -&-\ -&-\ -&-\ -&-\ -&-\ -&-\ -&-\ -&-\ -\\
\end{tabular}\\
\\
\begin{tabular}{r||r|r|r|r|r|r|r|r|r|r|}
&{\color{red}131}&{\color{red}132}&{\color{red}133}&{\color{red}134}&{\color{red}135}&{\color{red}136}&{\color{red}137}&{\color{red}138}&{\color{red}139}&{\color{red}140}\\
&{\color{red}$131^1$}&{\color{red}$2^23^111^1$}&{\color{red}$7^119^1$}&{\color{red}$2^167^1$}&{\color{red}$3^35^1$}&{\color{red}$2^317^1$}&{\color{red}$137^1$}&{\color{red}$2^13^123^1$}&{\color{red}$139^1$}&{\color{red}$2^25^17^1$}\\
{\color{green}2}&2 0&2 0&4 0&2 0&4 0&3 0&2 0&4 0&3 0&2 0\\
{\color{green}4}&2 0&4 0&5 0&5 0&8 0&4 0&2 0&6 0&3 0&6 0\\
{\color{green}6}&2 0&6 0&4 0&4 0&10 0&6 0&2 0&9 0&2 0&4 0\\
{\color{green}8}&2 0&4 0&4 0&4 0&8 0&4 0&2 0&8 0&2 0&4 0\\
{\color{green}10}&2 0&4 0&4 0&4 0&8 0&4 0&2 0&8 0&2 0&5 0\\
{\color{green}12}&2 0&4 0&4 0&4 0&8 0&4 0&2 0&8 0&2 0&4 0\\
{\color{green}14}&2 0&4 0&4 0&4 0&8 0&4 0&2 0&8 0&2 0&4 0\\
{\color{green}16}&2 0&4 0&4 0&4 0&8 0&4 0&2 0&8 0&2 0&4 0\\
{\color{green}18}&2 0&4 0&4 0&4 0&8 0&4 0&2 0&8 0&2 0&4 0\\
{\color{green}20}&2 0&4 0&4 0&4 0&8 0&4 0&2 0&8 0&2 0&4 0\\
{\color{green}22}&2 0&4 0&4 0&4 0&8 0&4 0&2 0&8 0&2 0&4 0\\
{\color{green}24}&2 0&4 0&4 0&4 0&8 0&4 0&2 0&8 0&2 0&4 0\\
{\color{green}26}&2 0&4 0&4 0&4 0&8 0&4 0&2 0&8 0&2 0&4 0\\
{\color{green}28}&2 0&4 0&4 0&4 0&8 0&4 0&2 0&8 0&2 0&4 0\\
{\color{green}30}&2 0&4 0&4 0&4 0&8 0&4 0&2 0&8 0&2 0&4 0\\
{\color{green}32}&-\ -&-\ -&-\ -&-\ -&-\ -&-\ -&-\ -&-\ -&-\ -&-\ -\\
{\color{green}34}&-\ -&-\ -&-\ -&-\ -&-\ -&-\ -&-\ -&-\ -&-\ -&-\ -\\
{\color{green}36}&-\ -&-\ -&-\ -&-\ -&-\ -&-\ -&-\ -&-\ -&-\ -&-\ -\\
{\color{green}38}&-\ -&-\ -&-\ -&-\ -&-\ -&-\ -&-\ -&-\ -&-\ -&-\ -\\
{\color{green}40}&-\ -&-\ -&-\ -&-\ -&-\ -&-\ -&-\ -&-\ -&-\ -&-\ -\\
\end{tabular}\\
\\
\begin{tabular}{r||r|r|r|r|r|r|r|r|r|r|}
&{\color{red}141}&{\color{red}142}&{\color{red}143}&{\color{red}144}&{\color{red}145}&{\color{red}146}&{\color{red}147}&{\color{red}148}&{\color{red}149}&{\color{red}150}\\
&{\color{red}$3^147^1$}&{\color{red}$2^171^1$}&{\color{red}$11^113^1$}&{\color{red}$2^43^2$}&{\color{red}$5^129^1$}&{\color{red}$2^173^1$}&{\color{red}$3^17^2$}&{\color{red}$2^237^1$}&{\color{red}$149^1$}&{\color{red}$2^13^15^2$}\\
{\color{green}2}&6 0&5 0&3 0&1 1&4 0&2 0&5 0&2 0&2 0&3 0\\
{\color{green}4}&5 0&4 0&4 0&6 1&5 0&5 0&13 0&2 0&2 0&9 0\\
{\color{green}6}&5 0&4 0&4 0&11 1&4 0&4 0&15 0&2 0&2 0&15 0\\
{\color{green}8}&4 0&4 0&4 0&13 1&4 0&4 0&13 0&2 0&2 0&19 0\\
{\color{green}10}&4 0&4 0&4 0&17 1&4 0&4 0&14 0&2 0&2 0&19 0\\
{\color{green}12}&4 0&4 0&4 0&19 1&4 0&4 0&14 0&2 0&2 0&21 0\\
{\color{green}14}&4 0&4 0&4 0&20 1&4 0&4 0&14 0&2 0&2 0&21 0\\
{\color{green}16}&4 0&4 0&4 0&23 1&4 0&4 0&14 0&2 0&2 0&25 0\\
{\color{green}18}&4 0&4 0&4 0&22 1&4 0&4 0&14 0&2 0&2 0&23 0\\
{\color{green}20}&4 0&4 0&4 0&23 1&4 0&4 0&14 0&2 0&2 0&23 0\\
{\color{green}22}&4 0&4 0&4 0&24 1&4 0&4 0&15 0&2 0&2 0&23 0\\
{\color{green}24}&4 0&4 0&4 0&23 1&4 0&4 0&14 0&2 0&2 0&24 0\\
{\color{green}26}&4 0&4 0&4 0&23 1&4 0&4 0&14 0&2 0&2 0&23 0\\
{\color{green}28}&4 0&4 0&4 0&24 1&4 0&4 0&14 0&2 0&2 0&24 0\\
{\color{green}30}&4 0&4 0&4 0&24 1&4 0&4 0&14 0&2 0&2 0&24 0\\
{\color{green}32}&-\ -&-\ -&-\ -&-\ -&-\ -&-\ -&14 0&-\ -&-\ -&24 0\\
{\color{green}34}&-\ -&-\ -&-\ -&-\ -&-\ -&-\ -&14 0&-\ -&-\ -&24 0\\
{\color{green}36}&-\ -&-\ -&-\ -&-\ -&-\ -&-\ -&14 0&-\ -&-\ -&-\ -\\
{\color{green}38}&-\ -&-\ -&-\ -&-\ -&-\ -&-\ -&-\ -&-\ -&-\ -&-\ -\\
{\color{green}40}&-\ -&-\ -&-\ -&-\ -&-\ -&-\ -&-\ -&-\ -&-\ -&-\ -\\
\end{tabular}\\
\\
\begin{tabular}{r||r|r|r|r|r|r|r|r|r|r|}
&{\color{red}151}&{\color{red}152}&{\color{red}153}&{\color{red}154}&{\color{red}155}&{\color{red}156}&{\color{red}157}&{\color{red}158}&{\color{red}159}&{\color{red}160}\\
&{\color{red}$151^1$}&{\color{red}$2^319^1$}&{\color{red}$3^217^1$}&{\color{red}$2^17^111^1$}&{\color{red}$5^131^1$}&{\color{red}$2^23^113^1$}&{\color{red}$157^1$}&{\color{red}$2^179^1$}&{\color{red}$3^153^1$}&{\color{red}$2^55^1$}\\
{\color{green}2}&3 0&3 0&5 0&4 0&5 0&2 0&2 0&6 0&2 0&3 0\\
{\color{green}4}&2 0&4 0&9 0&9 0&5 0&4 0&2 0&5 0&6 0&7 0\\
{\color{green}6}&2 0&5 0&9 0&10 0&4 0&4 0&2 0&5 0&5 0&8 0\\
{\color{green}8}&2 0&4 0&10 0&8 0&4 0&5 0&2 0&4 0&4 0&10 0\\
{\color{green}10}&2 0&4 0&8 0&8 0&4 0&4 0&2 0&4 0&4 0&8 0\\
{\color{green}12}&2 0&4 0&8 0&8 0&4 0&4 0&2 0&4 0&4 0&8 0\\
{\color{green}14}&2 0&4 0&8 0&8 0&4 0&4 0&2 0&4 0&4 0&8 0\\
{\color{green}16}&2 0&4 0&8 0&8 0&4 0&4 0&2 0&4 0&4 0&8 0\\
{\color{green}18}&2 0&4 0&8 0&8 0&4 0&4 0&2 0&4 0&4 0&8 0\\
{\color{green}20}&2 0&4 0&8 0&8 0&4 0&4 0&2 0&4 0&4 0&8 0\\
{\color{green}22}&2 0&4 0&8 0&8 0&4 0&4 0&2 0&4 0&4 0&8 0\\
{\color{green}24}&2 0&4 0&8 0&8 0&4 0&4 0&2 0&4 0&4 0&8 0\\
{\color{green}26}&2 0&4 0&8 0&8 0&4 0&4 0&2 0&4 0&4 0&8 0\\
{\color{green}28}&2 0&4 0&8 0&8 0&4 0&4 0&2 0&4 0&4 0&8 0\\
{\color{green}30}&2 0&4 0&8 0&8 0&4 0&4 0&2 0&4 0&4 0&8 0\\
{\color{green}32}&-\ -&-\ -&-\ -&-\ -&-\ -&-\ -&-\ -&-\ -&-\ -&-\ -\\
{\color{green}34}&-\ -&-\ -&-\ -&-\ -&-\ -&-\ -&-\ -&-\ -&-\ -&-\ -\\
{\color{green}36}&-\ -&-\ -&-\ -&-\ -&-\ -&-\ -&-\ -&-\ -&-\ -&-\ -\\
{\color{green}38}&-\ -&-\ -&-\ -&-\ -&-\ -&-\ -&-\ -&-\ -&-\ -&-\ -\\
{\color{green}40}&-\ -&-\ -&-\ -&-\ -&-\ -&-\ -&-\ -&-\ -&-\ -&-\ -\\
\end{tabular}\\
\\
\begin{tabular}{r||r|r|r|r|r|r|r|r|r|r|}
&{\color{red}161}&{\color{red}162}&{\color{red}163}&{\color{red}164}&{\color{red}165}&{\color{red}166}&{\color{red}167}&{\color{red}168}&{\color{red}169}&{\color{red}170}\\
&{\color{red}$7^123^1$}&{\color{red}$2^13^4$}&{\color{red}$163^1$}&{\color{red}$2^241^1$}&{\color{red}$3^15^111^1$}&{\color{red}$2^183^1$}&{\color{red}$167^1$}&{\color{red}$2^33^17^1$}&{\color{red}$13^2$}&{\color{red}$2^15^117^1$}\\
{\color{green}2}&4 0&4 0&3 0&1 0&3 0&3 0&2 0&2 0&3 0&6 0\\
{\color{green}4}&4 0&8 0&2 0&2 0&8 0&5 0&2 0&8 0&12 0&8 0\\
{\color{green}6}&4 0&10 0&2 0&2 0&8 0&4 0&2 0&10 0&8 0&8 0\\
{\color{green}8}&4 0&10 0&2 0&2 0&9 0&4 0&2 0&9 0&9 0&8 0\\
{\color{green}10}&4 0&10 0&2 0&2 0&8 0&4 0&2 0&8 0&8 0&9 0\\
{\color{green}12}&4 0&10 0&2 0&2 0&8 0&4 0&2 0&8 0&9 0&8 0\\
{\color{green}14}&4 0&10 0&2 0&2 0&8 0&4 0&2 0&8 0&8 0&8 0\\
{\color{green}16}&4 0&10 0&2 0&2 0&8 0&4 0&2 0&8 0&9 0&8 0\\
{\color{green}18}&4 0&10 0&2 0&2 0&8 0&4 0&2 0&8 0&9 0&8 0\\
{\color{green}20}&4 0&10 0&2 0&2 0&8 0&4 0&2 0&8 0&9 0&8 0\\
{\color{green}22}&4 0&10 0&2 0&2 0&8 0&4 0&2 0&8 0&9 0&8 0\\
{\color{green}24}&4 0&10 0&2 0&2 0&8 0&4 0&2 0&8 0&9 0&8 0\\
{\color{green}26}&4 0&10 0&2 0&2 0&8 0&4 0&2 0&8 0&9 0&8 0\\
{\color{green}28}&4 0&10 0&2 0&2 0&8 0&4 0&2 0&8 0&9 0&8 0\\
{\color{green}30}&4 0&10 0&2 0&2 0&8 0&4 0&2 0&8 0&9 0&8 0\\
{\color{green}32}&-\ -&-\ -&-\ -&-\ -&-\ -&-\ -&-\ -&-\ -&-\ -&-\ -\\
{\color{green}34}&-\ -&-\ -&-\ -&-\ -&-\ -&-\ -&-\ -&-\ -&-\ -&-\ -\\
{\color{green}36}&-\ -&-\ -&-\ -&-\ -&-\ -&-\ -&-\ -&-\ -&-\ -&-\ -\\
{\color{green}38}&-\ -&-\ -&-\ -&-\ -&-\ -&-\ -&-\ -&-\ -&-\ -&-\ -\\
{\color{green}40}&-\ -&-\ -&-\ -&-\ -&-\ -&-\ -&-\ -&-\ -&-\ -&-\ -\\
\end{tabular}\\
\\
\begin{tabular}{r||r|r|r|r|r|r|r|r|r|r|}
&{\color{red}171}&{\color{red}172}&{\color{red}173}&{\color{red}174}&{\color{red}175}&{\color{red}176}&{\color{red}177}&{\color{red}178}&{\color{red}179}&{\color{red}180}\\
&{\color{red}$3^219^1$}&{\color{red}$2^243^1$}&{\color{red}$173^1$}&{\color{red}$2^13^129^1$}&{\color{red}$5^27^1$}&{\color{red}$2^411^1$}&{\color{red}$3^159^1$}&{\color{red}$2^189^1$}&{\color{red}$179^1$}&{\color{red}$2^23^25^1$}\\
{\color{green}2}&5 0&2 0&2 0&5 0&6 0&4 0&4 0&4 0&3 0&1 0\\
{\color{green}4}&9 0&2 0&2 0&9 0&10 0&10 0&4 0&4 0&2 0&5 0\\
{\color{green}6}&12 0&2 0&2 0&8 0&12 0&12 0&4 0&4 0&2 0&7 0\\
{\color{green}8}&10 0&2 0&2 0&8 0&12 0&12 0&4 0&4 0&2 0&8 0\\
{\color{green}10}&8 0&2 0&2 0&8 0&13 0&13 0&4 0&4 0&2 0&8 0\\
{\color{green}12}&8 0&2 0&2 0&8 0&12 0&12 0&4 0&4 0&2 0&8 0\\
{\color{green}14}&8 0&2 0&2 0&8 0&12 0&12 0&4 0&4 0&2 0&9 0\\
{\color{green}16}&8 0&2 0&2 0&8 0&12 0&12 0&4 0&4 0&2 0&8 0\\
{\color{green}18}&8 0&2 0&2 0&8 0&12 0&12 0&4 0&4 0&2 0&8 0\\
{\color{green}20}&8 0&2 0&2 0&8 0&12 0&13 0&4 0&4 0&2 0&8 0\\
{\color{green}22}&8 0&2 0&2 0&8 0&12 0&12 0&4 0&4 0&2 0&8 0\\
{\color{green}24}&8 0&2 0&2 0&8 0&12 0&12 0&4 0&4 0&2 0&8 0\\
{\color{green}26}&8 0&2 0&2 0&8 0&12 0&12 0&4 0&4 0&2 0&8 0\\
{\color{green}28}&8 0&2 0&2 0&8 0&12 0&12 0&4 0&4 0&2 0&8 0\\
{\color{green}30}&8 0&2 0&2 0&8 0&12 0&12 0&4 0&4 0&2 0&-\ -\\
{\color{green}32}&-\ -&-\ -&-\ -&-\ -&-\ -&-\ -&-\ -&-\ -&-\ -&-\ -\\
{\color{green}34}&-\ -&-\ -&-\ -&-\ -&-\ -&-\ -&-\ -&-\ -&-\ -&-\ -\\
{\color{green}36}&-\ -&-\ -&-\ -&-\ -&-\ -&-\ -&-\ -&-\ -&-\ -&-\ -\\
{\color{green}38}&-\ -&-\ -&-\ -&-\ -&-\ -&-\ -&-\ -&-\ -&-\ -&-\ -\\
{\color{green}40}&-\ -&-\ -&-\ -&-\ -&-\ -&-\ -&-\ -&-\ -&-\ -&-\ -\\
\end{tabular}\\
\\
\begin{tabular}{r||r|r|r|r|r|r|r|r|r|r|}
&{\color{red}181}&{\color{red}182}&{\color{red}183}&{\color{red}184}&{\color{red}185}&{\color{red}186}&{\color{red}187}&{\color{red}188}&{\color{red}189}&{\color{red}190}\\
&{\color{red}$181^1$}&{\color{red}$2^17^113^1$}&{\color{red}$3^161^1$}&{\color{red}$2^323^1$}&{\color{red}$5^137^1$}&{\color{red}$2^13^131^1$}&{\color{red}$11^117^1$}&{\color{red}$2^247^1$}&{\color{red}$3^37^1$}&{\color{red}$2^15^119^1$}\\
{\color{green}2}&2 0&5 0&3 0&5 0&5 0&4 0&6 0&2 0&6 0&4 0\\
{\color{green}4}&3 0&10 0&5 0&6 0&4 0&10 0&6 0&2 0&12 0&9 0\\
{\color{green}6}&2 0&10 0&4 0&4 0&4 0&8 0&4 0&2 0&14 0&9 0\\
{\color{green}8}&2 0&8 0&4 0&4 0&4 0&11 0&4 0&2 0&8 0&8 0\\
{\color{green}10}&2 0&9 0&4 0&4 0&4 0&8 0&4 0&2 0&8 0&8 0\\
{\color{green}12}&2 0&8 0&4 0&4 0&4 0&8 0&4 0&2 0&8 0&8 0\\
{\color{green}14}&2 0&8 0&4 0&4 0&4 0&8 0&4 0&2 0&8 0&8 0\\
{\color{green}16}&2 0&8 0&4 0&4 0&4 0&8 0&4 0&2 0&8 0&8 0\\
{\color{green}18}&2 0&8 0&4 0&4 0&4 0&8 0&4 0&2 0&8 0&8 0\\
{\color{green}20}&2 0&8 0&4 0&4 0&4 0&8 0&4 0&2 0&8 0&8 0\\
{\color{green}22}&2 0&8 0&4 0&4 0&4 0&8 0&4 0&2 0&8 0&8 0\\
{\color{green}24}&2 0&8 0&4 0&4 0&4 0&8 0&4 0&2 0&8 0&8 0\\
{\color{green}26}&2 0&8 0&4 0&4 0&4 0&8 0&4 0&2 0&8 0&8 0\\
{\color{green}28}&2 0&8 0&4 0&4 0&4 0&8 0&4 0&2 0&8 0&8 0\\
{\color{green}30}&2 0&8 0&4 0&4 0&4 0&-\ -&4 0&2 0&8 0&8 0\\
{\color{green}32}&-\ -&-\ -&-\ -&-\ -&-\ -&-\ -&-\ -&-\ -&-\ -&-\ -\\
{\color{green}34}&-\ -&-\ -&-\ -&-\ -&-\ -&-\ -&-\ -&-\ -&-\ -&-\ -\\
{\color{green}36}&-\ -&-\ -&-\ -&-\ -&-\ -&-\ -&-\ -&-\ -&-\ -&-\ -\\
{\color{green}38}&-\ -&-\ -&-\ -&-\ -&-\ -&-\ -&-\ -&-\ -&-\ -&-\ -\\
{\color{green}40}&-\ -&-\ -&-\ -&-\ -&-\ -&-\ -&-\ -&-\ -&-\ -&-\ -\\
\end{tabular}\\
\\
\begin{tabular}{r||r|r|r|r|r|r|r|r|r|r|}
&{\color{red}191}&{\color{red}192}&{\color{red}193}&{\color{red}194}&{\color{red}195}&{\color{red}196}&{\color{red}197}&{\color{red}198}&{\color{red}199}&{\color{red}200}\\
&{\color{red}$191^1$}&{\color{red}$2^63^1$}&{\color{red}$193^1$}&{\color{red}$2^197^1$}&{\color{red}$3^15^113^1$}&{\color{red}$2^27^2$}&{\color{red}$197^1$}&{\color{red}$2^13^211^1$}&{\color{red}$199^1$}&{\color{red}$2^35^2$}\\
{\color{green}2}&2 0&4 0&3 0&3 0&5 0&3 0&3 0&5 0&3 0&5 0\\
{\color{green}4}&2 0&12 0&2 0&4 0&10 0&7 0&2 0&10 0&2 0&12 0\\
{\color{green}6}&2 0&18 0&2 0&4 0&9 0&11 0&2 0&14 0&2 0&11 0\\
{\color{green}8}&2 0&22 0&2 0&4 0&9 0&6 0&2 0&15 0&2 0&18 0\\
{\color{green}10}&2 0&24 0&2 0&4 0&8 0&7 0&2 0&19 0&2 0&12 0\\
{\color{green}12}&2 0&30 0&2 0&4 0&8 0&7 0&2 0&16 0&2 0&13 0\\
{\color{green}14}&2 0&26 0&2 0&4 0&8 0&7 0&2 0&16 0&2 0&12 0\\
{\color{green}16}&2 0&30 0&2 0&4 0&8 0&7 0&2 0&16 0&2 0&13 0\\
{\color{green}18}&2 0&30 0&2 0&4 0&8 0&7 0&2 0&16 0&2 0&12 0\\
{\color{green}20}&2 0&30 0&2 0&4 0&8 0&7 0&2 0&17 0&2 0&12 0\\
{\color{green}22}&2 0&32 0&2 0&4 0&8 0&7 0&2 0&16 0&2 0&12 0\\
{\color{green}24}&2 0&32 0&2 0&4 0&8 0&7 0&2 0&16 0&2 0&12 0\\
{\color{green}26}&2 0&30 0&2 0&4 0&8 0&7 0&2 0&16 0&2 0&12 0\\
{\color{green}28}&2 0&32 0&2 0&4 0&8 0&7 0&2 0&16 0&2 0&12 0\\
{\color{green}30}&2 0&-\ -&2 0&4 0&-\ -&-\ -&2 0&-\ -&2 0&-\ -\\
{\color{green}32}&-\ -&32 0&-\ -&-\ -&-\ -&-\ -&-\ -&-\ -&-\ -&-\ -\\
{\color{green}34}&-\ -&32 0&-\ -&-\ -&-\ -&-\ -&-\ -&-\ -&-\ -&-\ -\\
{\color{green}36}&-\ -&-\ -&-\ -&-\ -&-\ -&-\ -&-\ -&-\ -&-\ -&-\ -\\
{\color{green}38}&-\ -&-\ -&-\ -&-\ -&-\ -&-\ -&-\ -&-\ -&-\ -&-\ -\\
{\color{green}40}&-\ -&-\ -&-\ -&-\ -&-\ -&-\ -&-\ -&-\ -&-\ -&-\ -\\
\end{tabular}\\
\\
\begin{tabular}{r||r|r|r|r|r|r|r|r|r|r|}
&{\color{red}210}&{\color{red}216}&{\color{red}224}&{\color{red}240}&{\color{red}243}&{\color{red}245}&{\color{red}256}&{\color{red}289}&{\color{red}320}&{\color{red}343}\\
&{\color{red}$2^13^15^17^1$}&{\color{red}$2^33^3$}&{\color{red}$2^57^1$}&{\color{red}$2^43^15^1$}&{\color{red}$3^5$}&{\color{red}$5^17^2$}&{\color{red}$2^8$}&{\color{red}$17^2$}&{\color{red}$2^65^1$}&{\color{red}$7^3$}\\
{\color{green}2}&5 0&4 0&4 0&4 0&4 2&8 0&0 5&6 0&7 0&3 2\\
{\color{green}4}&11 0&8 0&8 0&12 0&9 2&16 0&9 5&9 0&19 0&4 2\\
{\color{green}6}&15 0&10 0&10 0&17 0&8 2&13 0&10 5&10 0&26 0&4 2\\
{\color{green}8}&16 0&8 0&8 0&21 0&8 2&14 0&13 5&10 0&28 0&4 2\\
{\color{green}10}&17 0&-\ -&-\ -&-\ -&-\ -&-\ -&-\ -&-\ -&-\ -&-\ -\\
{\color{green}12}&-\ -&-\ -&-\ -&-\ -&-\ -&-\ -&-\ -&-\ -&-\ -&-\ -\\
{\color{green}14}&-\ -&-\ -&-\ -&-\ -&-\ -&-\ -&-\ -&-\ -&-\ -&-\ -\\
{\color{green}16}&-\ -&-\ -&-\ -&-\ -&-\ -&-\ -&-\ -&-\ -&-\ -&-\ -\\
{\color{green}18}&-\ -&-\ -&-\ -&-\ -&-\ -&-\ -&-\ -&-\ -&-\ -&-\ -\\
{\color{green}20}&-\ -&-\ -&-\ -&-\ -&-\ -&-\ -&-\ -&-\ -&-\ -&-\ -\\
{\color{green}22}&-\ -&-\ -&-\ -&-\ -&-\ -&-\ -&-\ -&-\ -&-\ -&-\ -\\
{\color{green}24}&-\ -&-\ -&-\ -&-\ -&-\ -&-\ -&-\ -&-\ -&-\ -&-\ -\\
{\color{green}26}&-\ -&-\ -&-\ -&-\ -&-\ -&-\ -&-\ -&-\ -&-\ -&-\ -\\
{\color{green}28}&-\ -&-\ -&-\ -&-\ -&-\ -&-\ -&-\ -&-\ -&-\ -&-\ -\\
{\color{green}30}&-\ -&-\ -&-\ -&-\ -&-\ -&-\ -&-\ -&-\ -&-\ -&-\ -\\
{\color{green}32}&-\ -&-\ -&-\ -&-\ -&-\ -&-\ -&-\ -&-\ -&-\ -&-\ -\\
{\color{green}34}&-\ -&-\ -&-\ -&-\ -&-\ -&-\ -&-\ -&-\ -&-\ -&-\ -\\
{\color{green}36}&-\ -&-\ -&-\ -&-\ -&-\ -&-\ -&-\ -&-\ -&-\ -&-\ -\\
{\color{green}38}&-\ -&-\ -&-\ -&-\ -&-\ -&-\ -&-\ -&-\ -&-\ -&-\ -\\
{\color{green}40}&-\ -&-\ -&-\ -&-\ -&-\ -&-\ -&-\ -&-\ -&-\ -&-\ -\\
\end{tabular}\\
\\
\begin{tabular}{r||r|r|r|r|r|r|r|r|r|r|}
&{\color{red}361}&{\color{red}441}&{\color{red}512}&{\color{red}529}&{\color{red}625}&{\color{red}672}&{\color{red}729}&{\color{red}841}&{\color{red}961}&{\color{red}1028}\\
&{\color{red}$19^2$}&{\color{red}$3^27^2$}&{\color{red}$2^9$}&{\color{red}$23^2$}&{\color{red}$5^4$}&{\color{red}$2^53^17^1$}&{\color{red}$3^6$}&{\color{red}$29^2$}&{\color{red}$31^2$}&{\color{red}$2^2257^1$}\\
{\color{green}2}&8 1&6 4&5 2&9 1&7 0&10 0&5 0&11 0&11 1&5 0\\
{\color{green}4}&14 1&21 3&13 2&12 2&7 0&18 0&5 0&12 0&12 2&2 0\\
{\color{green}6}&12 1&27 4&10 2&10 1&7 0&18 0&5 0&12 0&12 1&2 0\\
{\color{green}8}&10 1&26 4&9 2&10 1&7 0&16 0&5 0&12 0&12 1&-\ -\\
{\color{green}10}&10 1&26 3&-\ -&10 2&-\ -&-\ -&5 0&-\ -&12 2&-\ -\\
{\color{green}12}&-\ -&28 4&-\ -&11 1&-\ -&-\ -&-\ -&-\ -&-\ -&-\ -\\
{\color{green}14}&-\ -&27 4&-\ -&10 1&-\ -&-\ -&-\ -&-\ -&-\ -&-\ -\\
{\color{green}16}&-\ -&-\ -&-\ -&11 2&-\ -&-\ -&-\ -&-\ -&-\ -&-\ -\\
{\color{green}18}&-\ -&-\ -&-\ -&-\ -&-\ -&-\ -&-\ -&-\ -&-\ -&-\ -\\
{\color{green}20}&-\ -&-\ -&-\ -&-\ -&-\ -&-\ -&-\ -&-\ -&-\ -&-\ -\\
{\color{green}22}&-\ -&-\ -&-\ -&-\ -&-\ -&-\ -&-\ -&-\ -&-\ -&-\ -\\
{\color{green}24}&-\ -&-\ -&-\ -&-\ -&-\ -&-\ -&-\ -&-\ -&-\ -&-\ -\\
{\color{green}26}&-\ -&-\ -&-\ -&-\ -&-\ -&-\ -&-\ -&-\ -&-\ -&-\ -\\
{\color{green}28}&-\ -&-\ -&-\ -&-\ -&-\ -&-\ -&-\ -&-\ -&-\ -&-\ -\\
{\color{green}30}&-\ -&-\ -&-\ -&-\ -&-\ -&-\ -&-\ -&-\ -&-\ -&-\ -\\
{\color{green}32}&-\ -&-\ -&-\ -&-\ -&-\ -&-\ -&-\ -&-\ -&-\ -&-\ -\\
{\color{green}34}&-\ -&-\ -&-\ -&-\ -&-\ -&-\ -&-\ -&-\ -&-\ -&-\ -\\
{\color{green}36}&-\ -&-\ -&-\ -&-\ -&-\ -&-\ -&-\ -&-\ -&-\ -&-\ -\\
{\color{green}38}&-\ -&-\ -&-\ -&-\ -&-\ -&-\ -&-\ -&-\ -&-\ -&-\ -\\
{\color{green}40}&-\ -&-\ -&-\ -&-\ -&-\ -&-\ -&-\ -&-\ -&-\ -&-\ -\\
\end{tabular}\\
\\
\begin{tabular}{r||r|r|r|r|r|r|r|r|r|r|}
&{\color{red}1369}&{\color{red}1681}&{\color{red}1225}&{\color{red}1331}&{\color{red}2187}&{\color{red}2197}&{\color{red}2209}&{\color{red}3125}&{\color{red}2401}&{\color{red}6561}\\
&{\color{red}$37^2$}&{\color{red}$41^2$}&{\color{red}$5^27^2$}&{\color{red}$11^3$}&{\color{red}$3^7$}&{\color{red}$13^3$}&{\color{red}$47^2$}&{\color{red}$5^5$}&{\color{red}$7^4$}&{\color{red}$3^8$}\\
{\color{green}2}&15 0&13 0&25 4&4 2&10 2&4 0&12 1&4 0&10 0&5 0\\
{\color{green}4}&12 0&15 0&42 4&4 2&8 2&4 0&14 1&4 0&8 0&-\ -\\
{\color{green}6}&11 0&-\ -&-\ -&-\ -&8 2&-\ -&12 2&-\ -&-\ -&-\ -\\
{\color{green}8}&-\ -&-\ -&-\ -&-\ -&-\ -&-\ -&-\ -&-\ -&-\ -&-\ -\\
{\color{green}10}&-\ -&-\ -&-\ -&-\ -&-\ -&-\ -&-\ -&-\ -&-\ -&-\ -\\
{\color{green}12}&-\ -&-\ -&-\ -&-\ -&-\ -&-\ -&-\ -&-\ -&-\ -&-\ -\\
{\color{green}14}&-\ -&-\ -&-\ -&-\ -&-\ -&-\ -&-\ -&-\ -&-\ -&-\ -\\
{\color{green}16}&-\ -&-\ -&-\ -&-\ -&-\ -&-\ -&-\ -&-\ -&-\ -&-\ -\\
{\color{green}18}&-\ -&-\ -&-\ -&-\ -&-\ -&-\ -&-\ -&-\ -&-\ -&-\ -\\
{\color{green}20}&-\ -&-\ -&-\ -&-\ -&-\ -&-\ -&-\ -&-\ -&-\ -&-\ -\\
\end{tabular}\\
\\
\begin{tabular}{r||r|r|r|r|}
&{\color{red}14641}&{\color{red}15625}&{\color{red}16807}&{\color{red}19683}\\
&{\color{red}$11^4$}&{\color{red}$5^6$}&{\color{red}$7^5$}&{\color{red}$3^9$}\\
{\color{green}2}&12 0&7 0&4 2&8 2\\
{\color{green}4}&-\ -&-\ -&-\ -&-\ -\\
{\color{green}6}&-\ -&-\ -&-\ -&-\ -\\
{\color{green}8}&-\ -&-\ -&-\ -&-\ -\\
{\color{green}10}&-\ -&-\ -&-\ -&-\ -\\
{\color{green}12}&-\ -&-\ -&-\ -&-\ -\\
{\color{green}14}&-\ -&-\ -&-\ -&-\ -\\
{\color{green}16}&-\ -&-\ -&-\ -&-\ -\\
{\color{green}18}&-\ -&-\ -&-\ -&-\ -\\
{\color{green}20}&-\ -&-\ -&-\ -&-\ -\\
\end{tabular}\\
\\

\bibliography{References}
\bibliographystyle{alpha}

\end{document}